\documentclass[a4paper,fleqn]{cas-sc}

\usepackage[authoryear,longnamesfirst]{natbib}
\usepackage{verbatim}
\usepackage{lscape}
\usepackage{amsmath}
\usepackage{amssymb}
\usepackage{esint}
\usepackage{multirow}
\usepackage{hhline}
\usepackage{tabularx}
\usepackage{mathtools}
\usepackage{float}
\usepackage{setspace}
\usepackage{mathtools}
\usepackage{enumitem}
\usepackage[dvipsnames]{xcolor}
\usepackage{algpseudocode}

\def\tsc#1{\csdef{#1}{\textsc{\lowercase{#1}}\xspace}}
\tsc{WGM}
\tsc{QE}

\newtheorem{theorem}{Theorem}
\newdefinition{remark}{Remark}
\newdefinition{algorithm}{Algorithm}
\newproof{proof}{Proof}

\begin{document}
\let\WriteBookmarks\relax
\def\floatpagepagefraction{1}
\def\textpagefraction{.001}

\shorttitle{Local Adaptive Kernel Methods - Poisson Equation}

\shortauthors{J. A. Reeger et~al.}

\title [mode = title]{Adaptivity in Local Kernel Based Methods for Approximating Solutions to the Poisson Equation}

\tnotemark[1]

\tnotetext[1]{This work was supported by the project ``Kernel Methods with Machine Learning and Adaptivity" sponsored by the Air Force Office of Scientific Research and by the Joint Directed Energy Transition Office.}

%

\author[1]{Jonah A. Reeger}

\cormark[1]


\ead{jonah.reeger.2@us.af.mil}


\credit{Conceptualization, Methodology, Software, Formal Analysis, Validation, Writing-Original Draft, Writing-Review \& Editing, Funding Acquisition}

\affiliation[1]{organization={Department of Mathematics and Statistics, Air Force Institute of Technology},
            addressline={2950 Hobson Way},
            city={Wright-Patterson Air Force Base},
            postcode={45305},
            state={OH},
            country={USA}}

\author[1]{Anders R. Johnson}


\ead{anders.johnson.4@us.af.mil}


\credit{Software, Formal Analysis, Investigation}


\author[3]{Shelby W. Woodrum}


\ead{shelby.woodrum@us.af.mil}


\credit{Investigation, Software, Validation}

\affiliation[3]{organization={Department of Mathematical Science, United States Air Force Academy},
            addressline={2345 Fairchild Drive},
            city={USAFA},
            postcode={80840},
            state={CO},
            country={USA}}

\cortext[1]{Corresponding Author}



\begin{abstract}
Expanding on the recent development of adaptive local kernel methods for approximating the action of linear operators, a local estimate of the error and an adaptive procedure for approximating solutions to the Poisson equation is developed. The error estimate is used in the midst of the adaptive procedure to determine locations in the solution domain where decreasing the spacing between nodes can decrease the error in the solution.  The approach described here is essentially ``meshless", with only local Delaunay triangulations leveraged as a convenience to provide information on the local geometry of the node set.  Once this information is utilized, it is discarded.  The experiments performed show close agreement between the error estimate and actual absolute forward error in the approximate solution along with a comparison to existing indicators for refinement.  The combination of the adaptive procedure and error estimate provides an automated technique to resolve localized features of a solution without the, often intractable, expense of uniform refinement across the entire solution domain.
\end{abstract}




\maketitle

\section{Introduction}
This article concerns the development of an error estimate and node refinement technique for approximating solutions to the Poisson equation that adapt to localized features of the solution.  These methods utilize local kernel-based interpolation, supplemented by polynomials, in the context of collocation.  An algorithm combining the error estimate and refinement techniques is described and used to demonstrate the performance of the estimate as an approximation to the actual forward error.  This work builds upon adaptive methods for approximating the action of linear operators in \cite{JAR2024}, which rely on a basis that includes shifts of a chosen conditionally positive definite kernel and a set of shifted monomials that span a space of polynomials up to a certain degree. The effectiveness of local kernel-based methods has been thoroughly demonstrated, often through the idea of radial basis function generated finite differences (RBF-FD), in the approximation of solutions to PDEs and the approximate evaluation of definite integrals \cite{FHNWW14,JGJAR2022,Hardy,Kansa,JAR2020,JAR2022,JARBF2016,JARBF2017,JARBFMLW2016,ASMV06}.

A hallmark of kernel-based methods is their ability to obtain tunable (often high) orders of accuracy, while allowing for nearly arbitrary configurations of the computational node set, including on domains with complex geometries, without maintaining a mesh \cite{BFNF2015b}.  A mesh provides information regarding connections between nodes, like edges in a triangulation. Many other approaches for solving PDEs require a mesh of the computational node set, often with restrictions on the quality of the mesh, or that the node set feature restrictive uniformity. These restrictions introduce often complicated approaches when considering the inclusion of ``$h$-adaptivity" in a process for solving a partial differential equation.  This ``$h$-adaptivity" refers to the increase or decrease in node density to improve accuracy and/or efficiency of the algorithm, and is often necessary when the solution to a PDE exhibits localized features that cannot be resolved without densely packed nodes near the features.  In this case, uniform refinement across the entire domain provides one approach to resolving the localized features; however, this tends to be intractable based on the availability of computational resources.  Instead, local increases in node density can allow for greater resolution only where it is necessary in the computational domain.  This is the approach taken here.

Locally adaptive approaches to solving PDEs require a method for determining locations in the solution domain, often the computational nodes, where the current approximation may be a poor representation of the solution.  In the context of kernel-based approximations, these refinement indicators come in many forms (see \cite{JSlakPhD} for a more detailed discussion), e.g.,
\begin{itemize}
    \item the size of a computed quantity, the magnitude of derivatives of (or approximations of derivatives of) computed quantities, or the detection of an interface \cite{DOBRAVEC202277,KosekSarler,DAVYDOV2011287,OANH2017474,OANH2022123}, or
    \item a norm of the residual that measures how well the current solution satisfies a discrete version of the governing equations \cite{LiZhaiWengFeng,TothDuster2023,9854342}.
\end{itemize}
The refinement indicator developed and analyzed in this work is instead an error estimator similar to the one used in \cite{JAR2024}, that relies on the difference of two separate approximations, exhibiting different orders of accuracy, in the spirit of the adaptive trapezoidal rule for numerical quadrature (see, e.g., \cite{atkinson}).

Without stringent restrictions on the uniformity of the node set or the need to maintain a mesh, local kernel-based approximations open up opportunities to explore novel strategies for introducing nodes to the computational node set (in locations identified by the refinement indicator).   Therefore, a strategy is introduced that leverages geometric information about the node set from locally constructed Delaunay triangulations, then immediately discards them.

A generic mathematical description of the problem, that applies to more than just the Poisson equation,  is presented in section \ref{sec:MathDescription}.  Then, local kernel-based approximations are introduced in section \ref{sec:Interpolation} and applied to the process of solving PDEs in section \ref{sec:PDESolution}. Next, the error estimator that is used as a refinement indicator in the adaptive approach proposed here is developed in sections \ref{sec:ErrorApprox} and \ref{sec:Error_Estimation}.  This error estimator is then utilized in the adaptive algorithm presented in section \ref{sec:Implementation} with experimental computational results in section \ref{sec:Experiments}.  Finally, conclusions are given in section \ref{sec:Conclusions}.  A Matlab implementation of the method is available at Matlab Central's File Exchange \cite{AdaptiveRBFPoissonCodeMatlab}.

\section{Mathematical Description of the Problem and Some Notation} \label{sec:MathDescription}

Consider the problem of determining the solution $u:\Omega\mapsto\mathbb{R}$, $\Omega\subset\mathbb{R}^{d}$, of
\begin{align}
    (\mathcal{L}u)(\mathbf{x})=f(\mathbf{x}),\quad \mathbf{x}\in\Omega \label{eq:PDEexact}
\end{align}
where $\mathcal{L}$ is a linear operator and $f:\Omega\mapsto\mathbb{R}$.  Multi-indices (see, e.g., \cite{hunter2001applied} section 11.1), which are ordered tuples that are here written as vectors in $\mathbb{Z}^{d}$, are used for convenience in describing ideas that apply regardless of the value of $d$.  For instance, employing multi-index notation, the action of $\mathcal{L}$ is defined by
\begin{align}
    (\mathcal{L}u)(\mathbf{x}) = \sum\limits_{|\boldsymbol{\alpha}'|\leq m}c_{\boldsymbol{\alpha}'}(\mathbf{x})(\partial^{\boldsymbol{\alpha}'}u)(\mathbf{x}),\nonumber
\end{align}
with $c_{\boldsymbol{\alpha}'}:\Omega\to\mathbb{R}$.   Note that given a vector $\mathbf{z}\in\mathbb{R}^{\eta}$, the notation $[\mathbf{z}]_{j}$, $j=1,2,\ldots,\eta$, will indicate the $j\textsuperscript{th}$ component of $\mathbf{z}$.  With this in mind, the multi-index $\boldsymbol\alpha\in\mathbb{Z}^{d}$ satisfies $[\boldsymbol{\alpha}]_{j}\geq0$, $j=1,2,\ldots,d$ and has order $|\boldsymbol\alpha|\vcentcolon=\sum_{j=1}^{d}[\boldsymbol{\alpha}]_{j}$.  The factorial of a multi-index is $\boldsymbol{\alpha}!\vcentcolon=\prod_{j=1}^{d}\left([\boldsymbol{\alpha}]_{j}!\right)$ and a partial ordering can be placed on multi-indices $\boldsymbol{\alpha},\boldsymbol{\alpha}'\in\mathbb{Z}^{d}$ by, for instance, $\boldsymbol{\alpha}'\leq\boldsymbol{\alpha}$ when $[\boldsymbol{\alpha}']_{j}\leq[\boldsymbol{\alpha}]_{j}$, for all $j=1,2,\ldots,d$. When considering dimension $d$, there are $M_{d,m}\vcentcolon=((m+d)!)(m!d!)$ unique multi-indices up to order $m$.  Further, partial derivatives can be expressed concisely with this notation as
\begin{align}           \partial^{\boldsymbol{\alpha}}\vcentcolon=\prod\limits_{j=1}^{d}\frac{\partial^{[\boldsymbol{\alpha}]_{j}}}{\partial [\mathbf{x}]_{j}^{[\boldsymbol{\alpha}]_{j}}},\nonumber
\end{align}
which will be useful when deriving the error estimate used in the adaptive refinement method developed herein.
In certain cases, it will be helpful to index the set of all multi-indices up to a given order.  For instance, let $\{\boldsymbol{\alpha}_{l}\}_{l=1}^{M_{d,m}} = \{\boldsymbol{\alpha}\in\mathbb{Z}^{d}\vcentcolon\lvert\boldsymbol{\alpha}\rvert\leq m\}$, where $\lvert\boldsymbol{\alpha}_{l'}\rvert\leq\lvert\boldsymbol{\alpha}_{l}\rvert$ when $l'\leq l$, and so that all multi-indices in the set of the same order are sorted lexicographically.

It is useful to think of the coefficients $c_{\boldsymbol{\alpha}'}(\mathbf{x})$ being defined piecewise, and analogously $f$ as a piecewise function, to simplify discussion.  That is, the action of $\mathcal{L}$ on $u(\mathbf{x})$ will depend on the location of $\mathbf{x}$ within $\Omega$.  On the interior of $\Omega$, the Laplace operator,
\begin{align}
   (\mathcal{L}u)(\mathbf{x}) = (\Delta u)(\mathbf{x}) = \sum\limits_{j=1}^{d}\left(\partial^{2\mathbf{e}_{j}}u\right)(\mathbf{x}),\nonumber
\end{align}
will be applied in all numerical experiments, with $\mathbf{e}_{j}\in\mathbb{R}^{d}$ the $j\textsuperscript{th}$ column of the $d\times d$ identity matrix.  In this case, $c_{\boldsymbol{\alpha}'}(\mathbf{x})=1$ if $\boldsymbol{\alpha}'=2\mathbf{e}_{i}$ and $0$ otherwise.   On the boundary of the domain, $\partial\Omega$, certain linear boundary conditions will apply and there the action of $\mathcal{L}$ may be, e.g.,
\begin{itemize}
    \item the identity map, with $c_{\boldsymbol{\alpha}'}(\mathbf{x})=1$ if $\boldsymbol{\alpha}'=\mathbf{0}_{d}$ (the zero vector in $\mathbb{R}^{d}$) and $0$ otherwise, so that
\begin{align}
    (\mathcal{L}u)(\mathbf{x}) = u(\mathbf{x}) \nonumber
\end{align}
in the case of Dirichlet boundary conditions, or
\item a directional derivative, with $c_{\boldsymbol{\alpha}'}(\mathbf{x})=[\mathbf{n}(\mathbf{x})]_{i}$ if $\boldsymbol{\alpha}'=\mathbf{e}_{i}$ and $0$ otherwise, so that
\begin{align}
    (\mathcal{L}u)(\mathbf{x}) = \sum\limits_{i=1}^{d}[\mathbf{n}(\mathbf{x})]_{i}\left(\frac{\partial}{\partial [\mathbf{x}]_{i}}u\right)(\mathbf{x}), \nonumber
\end{align}
where $\mathbf{n}:\mathbb{R}^{d}\mapsto\mathbb{R}^{d}$ satisfying $\lVert\mathbf{n}(\mathbf{x})\rVert_{2}=1$, in the case of Neumann boundary conditions (used in the case of test function $u_{4}$ in section \ref{sec:Experiments}).
\end{itemize}
In any case, only a certain subset of the coefficients $c_{\boldsymbol{\alpha}'}$ are nonzero.

\section{Local Kernel Approximations Via Interpolation} \label{sec:Interpolation}

On the domain $\Omega$, an approximate value of $u$ is sought for each element in a set $\mathcal{S}=\left\{\mathbf{x}_{k}\right\}_{k=1}^{N}$ of $N$ unique computational nodes.  To each point $\mathbf{x}_{k}\in\mathcal{S}$ associate a set $\mathcal{N}_{k,n_{k}}=\left\{\mathbf{x}_{k,j}\right\}_{j=1}^{n_{k}}$ containing the $n_{k}\ll N$ points in $\mathcal{S}$ nearest to $\mathbf{x}_{k}$.  The nodes in this set are assumed to be ordered so that $\lVert\mathbf{x}_{k,i}-\mathbf{x}_{k}\rVert_{2}\leq\lVert\mathbf{x}_{k,j}-\mathbf{x}_{k}\rVert_{2}$, when $i<j$.  The approximations are then chosen to be local, so that the function $u$ is represented by a different approximating function, $s_{k,n_{k},m}[u]$, at each $\mathbf{x}_{k}$.  Here, and elsewhere in this work, including  a function in brackets (e.g., $s_{k,n,m}[u]$) when describing, for example, an interpolating function or set of coefficients explicitly indicates that these objects are dependent on the function in brackets.  The approximating function will be a linear combination of basis functions that match $u$ exactly at each point in the set $\mathcal{N}_{k,n_{k}}$. While the value of $n_{k}$ need not be the same for each $k$, e.g., with larger values near domain boundaries beneficial to mitigate errors similar to those introduced by the Runge phenomenon \cite{JARBF2017}, the computational experiments will choose $n_{k}=n$ for all $k$ with no apparent adverse effects. Further, the accuracy of the local approximation of $u$ near $\mathbf{x}_{k}$ relies on the value $h_{k,n}$, a characteristic spacing for the set $\mathcal{N}_{k,n}$.  This value could be the maximum or minimum (used here) distance from $\mathbf{x}_{k}$ to any other point in $\mathcal{N}_{k,n}$ or
\begin{align}
    \max\limits_{j=1,2,\ldots,n}\min\limits_{\begin{array}{c}i=1,2,\ldots,n\\i\neq j\end{array}}\lVert\mathbf{x}_{i}-\mathbf{x}_{j}\rVert_{2},\nonumber
\end{align}
which approximates the so-called ``fill distance" when node sets are quasi-uniformly spaced. In any case, the value of $h_{k,n}$ should decrease when new computational nodes are added to $\mathcal{S}$ near $\mathbf{x}_{k}$.  While the points in $\mathcal{N}_{k,n}$ need not be those nearest to $\mathbf{x}_{k}$, other choices may increase the value of $h_{k,n}$ with impacts on the accuracy in the approximation of $u$ at $\mathbf{x}_{k}$.

\subsection{Expressing the Interpolant}

The local approximations chosen here are linear combinations of (conditionally) positive definite kernels, $\varphi$, evaluated at the set of center points $\mathcal{N}_{k,n}$, i.e.,
\begin{align}
\phi_{k,n,j}(\mathbf{x})\vcentcolon=\varphi\left(\left\lVert \mathbf{x}-\mathbf{x}_{k,j}\right\rVert_{2}\right), j=1,2,\ldots,n.\nonumber
\end{align}
In the conditionally positive definite case, multivariate polynomial terms up to total degree $m$ are introduced to supplement the kernel basis as a guarantee for the existence of unique solutions to the resulting interpolation problem \cite{HW2005}.   The existence and uniqueness result requires satisfaction of the following conditions:
\begin{enumerate}
    \item the kernel $\varphi$ is conditionally positive-definite of order $m+1$, and
    \item the set $\mathcal{N}_{k,n}$ is unisolvent on the space, $\mathbb{P}_{m}^{d}$, of $d$-variate polynomials up to degree $m$.
\end{enumerate}

A (shifted) polynomial basis is used to improve numerical stability and consists of the functions (for all $\boldsymbol{\alpha}$ with $\lvert\boldsymbol{\alpha}\rvert\leq m$)
\begin{align}
    \pi_{k,\boldsymbol{\alpha}}(\mathbf{x})=(\mathbf{x}-\mathbf{x}_{k})^{\boldsymbol\alpha}\vcentcolon=\prod\limits_{j=1}^{d}[(\mathbf{x}-\mathbf{x}_{k})]_{j}^{[\boldsymbol{\alpha}]_{j}}. \nonumber
\end{align}
Note that multiple subscripting, with, e.g., $k$, $n$, $m$ and $\boldsymbol{\alpha}$, is used here and throughout this work to remind the reader of the dependence of certain quantities on the many parameters.

The interpolant is constructed as
\begin{align}
s_{k,n,m}[u](\mathbf{x})\vcentcolon=\sum_{j=1}^{n}\lambda_{k,n,m,j}[u]\phi_{k,n,j}\left(\mathbf{x}\right)+\sum_{|\boldsymbol{\alpha}|\leq m}\gamma_{k,n,m,\boldsymbol{\alpha}}[u]\pi_{k,\boldsymbol{\alpha}}(\mathbf{x}).\nonumber
\end{align}
with the coefficients $\lambda_{k,n,m,j}[u]$, $j=1,2,\ldots,n$, and $\gamma_{k,n,m,\boldsymbol{\alpha}}[u]$, $|\boldsymbol{\alpha}|\leq m$, chosen to first satisfy the $n$ interpolation conditions
\begin{align}
    s_{k,n,m}[u](\mathbf{x}_{k,j})=u(\mathbf{x}_{k,j}),\quad j=1,2,\ldots,n.\label{eq:interpcond}
\end{align}
This system of equations is under-determined, requiring a further $M_{d,m}$ equations to uniquely define all $n+M_{d,m}$ coefficients.  For the remainder of this work, only the kernel $\varphi(r)=r^{\rho}$, $0<\rho\in\mathbb{Z}$ and odd, will be considered.  When $m\geq\lceil\rho/2\rceil-1$, imposing
\begin{align}
    \sum_{j=1}^{n}\lambda_{k,n,m,j}[u]\pi_{k,\boldsymbol{\alpha}}(\mathbf{x}_{k,j})=0,\label{eq:CPD}
\end{align}
for all multi-indicies with $|\boldsymbol{\alpha}|\leq m$ guarantees the existence of a unique set of coefficients when interpolating any function $u$ as long as conditions 1 and 2 are also satisfied (see, e.g., \cite{HW2005}).

The equations \eqref{eq:interpcond} and \eqref{eq:CPD} can be summarized in the system of linear equations
\begin{align}
S_{k,n,m}\left[\begin{array}{c}\boldsymbol{\lambda}_{k,n,m}[u] \\ \boldsymbol{\gamma}_{k,n,m}[u]\end{array}\right]=\left[\begin{array}{c}\mathbf{u}_{k,n} \\ \mathbf{0}_{M_{d,m}}\end{array}\right].\nonumber
\end{align}
Here, $[\boldsymbol{\lambda}_{k,n,m}[u]]_{j}=\lambda_{k,n,m,j}[u]$, $[\boldsymbol{\gamma}_{k,n,m}[u]]_{l}=\gamma_{k,n,m,\boldsymbol{\alpha}_{l}}[u]$ and
$[\mathbf{u}_{k,n}]_{j}=u(\mathbf{x}_{k,j})$.
Further, the partitioned matrix
\begin{align}
S_{k,n,m}=\left[\begin{array}{cc}\Phi_{k,n} & P_{k,n,m}\\P_{k,n,m}^{T} & 0_{M_{d,m}\times M_{d,m}}\end{array}\right]\label{eq:InterpLinSys}
\end{align}
has blocks with entries, $[\Phi_{k,n}]_{ij}=\phi_{k,n,j}\left(\mathbf{x}_{k,i}\right)$ and $[P_{k,n,m}]_{il}=\pi_{k,\boldsymbol{\alpha}_{l}}(\mathbf{x}_{k,i})$, all with $i,j=1,2,\ldots,n$ and $l=1,2,\ldots,M_{d,m}$.

Alternatively, the interpolant can be expressed in a more convenient form for the presentation of theoretical results. That is,
\begin{align}
 s_{k,n,m}[u](\mathbf{x})= \sum\limits_{j=1}^{n}\psi_{k,n,m,j}(\mathbf{x})u(\mathbf{x}_{k,j}),\label{eq:cardinterp}
\end{align}
where the new set of functions satisfy the cardinal property
\begin{align}
    \psi_{k,n,m,j}(\mathbf{x}_{k,i})=\left\{\begin{array}{cc} 1 & i=j \\ 0 & i\neq j\end{array}\right..\nonumber
\end{align}
Equating the two forms of the interpolant produces the relationship
\begin{align}
    S_{k,n,m}\left[\begin{array}{c}\boldsymbol{\psi}_{k,n,m}(\mathbf{x})\\\boldsymbol{\xi}_{k,n,m}(\mathbf{x})\end{array}\right]=\left[\begin{array}{c}\boldsymbol{\phi}_{k,n}(\mathbf{x})\\\boldsymbol{\pi}_{k,m}(\mathbf{x})\end{array}\right]\label{eq:cardinalrelationship}
\end{align}
with $[\boldsymbol{\phi}_{k,n}(\mathbf{x})]_{j}=\phi_{k,n,j}(\mathbf{x})$ and  $[\boldsymbol{\pi}_{k,m}(\mathbf{x})]_{l}=\pi_{k,\boldsymbol{\alpha}_{l}}(\mathbf{x})$ and $[\boldsymbol{\psi}_{k,n,m}(\mathbf{x})]_{j}=\psi_{k,n,m,j}(\mathbf{x})$ vectors with the basis functions as entries.  The vector $\boldsymbol{\xi}_{k,n,m}(\mathbf{x})$ is not needed when forming the interpolant; however, its entries can be interpreted as Lagrange multipliers in a related quadratic program (see, e.g., \cite{VB2019})

\section{Approximating Solutions to the PDE} \label{sec:PDESolution}

When expressed in cardinal form it is clear that the interpolation conditions \eqref{eq:interpcond} are immediately satisfied.  However, this only indicates that the interpolant matches $u$ at the points in $\mathcal{N}_{k,n}$.  To understand the behavior of the interpolant at other values of $\mathbf{x}$, it is shown in \cite{JAR2024} that if conditions 1 and 2 are satisfied and $u$ has continuous mixed partial derivatives up to order $m+1$ in a convex neighborhood of $\mathbf{x}_{k}$ containing $\mathcal{N}_{k,n}$ and $\mathbf{x}$, then the point-wise error in the kernel based interpolant $s_{k,n,m}[u]$ is
\begin{align}
    s_{k,n,m}[u](\mathbf{x})-u(\mathbf{x})=s_{k,n,m}[R_{m}[u]](\mathbf{x})-R_{m}[u](\mathbf{x})\label{eq:interperror}
\end{align}
where $R_{m}[u](\mathbf{x})$ is the Taylor remainder term expressible as \cite{CartanDC,TrenchRA}
\begin{align}
R_{m}[u](\mathbf{x})=\sum\limits_{|\boldsymbol{\alpha}|=m+1}\frac{m+1}{\boldsymbol\alpha!}I_{k,\boldsymbol{\alpha}}[u](\mathbf{x})(\mathbf{x}-\mathbf{x}_{k})^{\boldsymbol\alpha},\nonumber
\end{align}
with
\begin{align}
    I_{k,\boldsymbol{\alpha}}[u](\mathbf{x})=\int\limits_{0}^{1}\partial^{\boldsymbol\alpha}u(\mathbf{y})|_{\mathbf{y}=\mathbf{x}_{k}+t(\mathbf{x}-\mathbf{x}_{k})}(1-t)^{|\boldsymbol{\alpha}|-1}dt.\nonumber
\end{align}
This leads to a convenient expression for the action of $\mathcal{L}$ on $ s_{k,n,m}[u](\mathbf{x})$, which can be exploited to approximate $u(\mathbf{x}_{k})$.
\begin{theorem} \label{thm:Lapprox}
     Suppose that conditions 1 and 2 are satisfied and that $u$ has continuous mixed partial derivatives up to order $m+1$ in a convex neighborhood of $\mathbf{x}_{k}$ containing $\mathcal{N}_{k,n}$ and $\mathbf{x}$, then
    \begin{align}
    \sum\limits_{j=1}^{n}(\mathcal{L}\psi_{k,n,m,j})(\mathbf{x}_{k})u(\mathbf{x}_{k,j})=f(\mathbf{x}_{k})+(\mathcal{L}s_{k,n,m}[R_{m}[u]])(\mathbf{x}_{k}).\label{eq:PDExkreduced}
    \end{align}
\end{theorem}
\begin{proof}
This follows directly from equation 11 of \cite{DS18} by noting that $(\mathcal{L}u)(\mathbf{x}_{k})=f(\mathbf{x}_{k})$.
\end{proof}

Notice that since $\mathcal{N}_{k,n}\subset\mathcal{S}$, the equations \eqref{eq:PDExkreduced} for $k=1,2,\ldots,N$ can be combined into a system of linear equations $D_{n,m}\mathbf{u}=\mathbf{f}+\mathbf{r}_{n,m}$ with $[\mathbf{u}]_{k}=u(x_{k})$, $[\mathbf{f}]_{k}=f(x_{k})$, $[\mathbf{r}_{n,m}]_{k}=(\mathcal{L}s_{k,n,m}[R_{m}[u]])(\mathbf{x}_{k})$ and
\begin{align}
    [D_{n,m}]_{ki} = \left\{\begin{array}{cc}(\mathcal{L}\psi_{k,n,m,j})(\mathbf{x}_{k}), & \exists \mbox{ }j\in\{1,2,\ldots,n\}\mbox{ such that } \mathbf{x}_{i}=\mathbf{x}_{k,j} \\ 0, & \mbox{otherwise}\end{array}\right..\nonumber
\end{align}
As long as $D_{n,m}$ is invertible, which is not guaranteed but often realized in practice, the values of $u(\mathbf{x}_{k})$ can be written $\mathbf{u}=D_{n,m}^{-1}\mathbf{f}+D_{n,m}^{-1}\mathbf{r}_{n,m}$.  Define $\mathbf{u}_{n,m}=D_{n,m}^{-1}\mathbf{f}$ so that $[\mathbf{u}_{n,m}]_{k}$ is an approximation for $u(\mathbf{x}_{k})$ with error given by
\begin{align}
    [\mathbf{u}-\mathbf{u}_{n,m}]_{k}=[D_{n,m}^{-1}\mathbf{r}_{n,m}]_{k}.\label{eq:forwarderror}
\end{align}

\section{Approximating the Error} \label{sec:ErrorApprox}

To understand the suitability of $\mathbf{u}_{n,m}$ as an approximation for $\mathbf{u}$ it is useful to analyze the entries of $D_{n,m}^{-1}\mathbf{r}_{n,m}$.  The next two sections analyze the terms in this product and their relationship to $h_{k,n}$.

\subsection{Applying the Operator to the Cardinal Basis} \label{sec:CardinalApprox}

The entries of $\mathbf{r}_{n,m}$ are given by
\begin{align}
 \left(\mathcal{L}s_{k,n,m}[R_{m}[u]]\right)(\mathbf{x}_{k})= \sum\limits_{|\boldsymbol{\alpha}|=m+1}\sum\limits_{j=1}^{n}\left(\mathcal{L}\psi_{k,n,m,j}\right)(\mathbf{x}_{k})\frac{m+1}{\boldsymbol{\alpha}!}I_{k,\boldsymbol{\alpha}}[u](\mathbf{x}_{k,j})(\mathbf{x}_{k,j}-\mathbf{x}_{k})^{\boldsymbol{\alpha}}.\nonumber
\end{align}
Expressing each $\mathbf{x}_{k,j}\in\mathcal{N}_{k,n}$ as $\mathbf{x}_{k,j}=\mathbf{x}_{k}+h_{k,n}\mathbf{v}_{k,j}$, where $\mathbf{v}_{k,j}$ encodes the direction from $\mathbf{x}_{k}$ to $\mathbf{x}_{k,j}$, first reveals
\begin{align}
    \left(\mathbf{x}_{k,j}-\mathbf{x}_{k}\right)^{\boldsymbol{\alpha}}=h_{k,n}^{|\boldsymbol{\alpha}|}\mathbf{v}_{k,j}^{\boldsymbol{\alpha}}.\nonumber
\end{align}
Further, suppose that conditions 1 and 2 are satisfied.  Then from lemma 3.2 of \cite{Iske}
\begin{align}
    (\mathcal{L}\psi_{k,n,m,j})(\mathbf{x}_{k})=O(h_{k,n}^{-\lvert\boldsymbol{\alpha}_{k,*}\rvert}),\nonumber
\end{align}
as $h_{k,n}\to0$, with $\lvert\boldsymbol{\alpha}_{k,*}\rvert$ the largest $\lvert\boldsymbol{\alpha}'\rvert$ such that $c_{\boldsymbol{\alpha}'}(\mathbf{x}_{k})\neq0$.  Therefore,
\begin{align}
    [\mathbf{r}_{n,m}]_{k} = O(h_{k,n}^{m+1-\lvert\boldsymbol{\alpha}_{k,*}\rvert})
\end{align}

\begin{remark} \label{rem:orientedlimit}
    The limit as $h_{k,n}\to0$ assumes that the orientation of the nodes in $\mathcal{N}_{k,n}$ is maintained.  That is, in the definition $\mathbf{x}_{k,j}=\mathbf{x}_{k}+h_{k,n}\mathbf{v}_{k,j}$ the vector $\mathbf{v}_{k,j}$ is constant with respect $h_{k,j}$.  However, in practice adaptive techniques, like the one described in the following sections, often do not rely on such uniform refinement, but still realize improvements in accuracy with decreases in $h_{k,n}$ that are similarly dependent on the polynomial order $m$.  This is demonstrated in section \ref{sec:Experiments} and in \cite{JAR2024}.
\end{remark}

\begin{remark}
    Consider the matrix
\begin{align}
    \tilde{S}_{k,n,m}=&\left[\begin{array}{cc} h_{k,n}^{-\rho}I_{n} & 0_{n\times M_{d,m}} \\ 0_{M_{d,m}\times n} & H_{k,n,m}^{-1}\end{array}\right]S_{k,n,m}\left[\begin{array}{cc} I_{n} & 0_{n\times M_{d,m}} \\ 0_{M_{d,m}\times n} & h_{k,n}^{\rho}H_{k,n,m}^{-1}\end{array}\right] \nonumber\\
    =& \left[\begin{array}{cc}h_{k,n}^{-\rho}\Phi_{k,n} & h_{k,n}^{-\rho}P_{k,n,m}h_{k,n}^{\rho}H_{k,n,m}^{-1}\\H_{k,n,m}^{-1}P_{k,n,m}^{T} & 0_{M_{d,m}\times M_{d,m}}\end{array}\right], \nonumber
\end{align}
where
\begin{align}
    H_{k,n,m}=\left[\begin{array}{cccc}
    h_{k,n}^{\lvert\boldsymbol{\alpha}_{1}\rvert} & & &\\
     & h_{k,n}^{\lvert\boldsymbol{\alpha}_{2}\rvert} & & \\
     & & \ddots & \\
     & & & h_{k,n}^{\lvert\boldsymbol{\alpha}_{M_{d,m}}\rvert}\end{array}\right].\nonumber
\end{align}
The system of equations
\begin{align}
\tilde{S}_{k,n,m}\left[\begin{array}{c}(\mathcal{L}\boldsymbol{\psi}_{k,n,m})(\mathbf{x}_{k})\\h_{k,n}^{-\rho}H_{k,n,m}(\mathcal{L}\boldsymbol{\xi}_{k,n,m})(\mathbf{x}_{k})\end{array}\right]=\left[\begin{array}{c}h_{k,n}^{-\rho}(\mathcal{L}\boldsymbol{\phi}_{k,n})(\mathbf{x}_{k})\\H_{k,n,m}^{-1}(\mathcal{L}\boldsymbol{\pi}_{k,m})(\mathbf{x}_{k})\end{array}\right],\label{eq:weightrelationshipscaled}
\end{align}
which is equivalent to \eqref{eq:cardinalrelationship} evaluated at $\mathbf{x}_{k}$, is solved to determine $(\mathcal{L}\boldsymbol{\psi}_{k,n,m})(\mathbf{x}_{k})$ to promote numerical stability \cite{Iske}.
\end{remark}

\subsection{The Inverse of $D_{n,m}$} \label{sec:MatInverse}

The results of section \ref{sec:CardinalApprox} further indicate that in $D_{n,m}$ the nonzero entries in the $k\textsuperscript{th}$ row are are $O(h_{k,n}^{-\lvert\boldsymbol{\alpha}_{k,*}\rvert})$.  Although $D_{n,m}$ is sparse, $D_{n,m}^{-1}$, when it exists, is full so that every entry of the forward error vector is a linear combination of all entries of $\mathbf{r}_{n,m}$.  Still, the goal here is to create a method that benefits from local refinement.  Ideally, local decreases in $h_{k,n}$ should translate to decreases in the modulus of the $k\textsuperscript{th}$ entry of the forward error vector, even when $h_{j,n}$, $j\neq k$, does not change.  Introducing the $N\times N$ diagonal matrix $G_{n}$ with diagonal entries $[G_{n}]_{kk} = h_{k,n}^{\lvert\boldsymbol{\alpha}_{k,*}\rvert}$, notice that
\begin{align}
    D_{n,m}^{-1} =  (G_{n}D_{n,m})^{-1}G_{n}.\nonumber
\end{align}
Further, combing the definition of $G_{n}$ with the results of section \ref{sec:CardinalApprox}, $[G_{n}\mathbf{r}_{n,m}]_{k}=O(h_{k,n}^{m+1})$ as $h_{k,n}\to0$

Given the dependence of $[G_{n}\mathbf{r}_{n,m}]_{k}$ on $h_{k,n}$ and the desire to construct a method that decreases the error in $[\mathbf{u}_{n,m}]_{k}$ by local changes in the density of the nodes, it would be beneficial to demonstrate that $\lvert[(G_{n}D_{n,m})^{-1}]_{ki}\rvert$ decreases rapidly as $\lVert\mathbf{x}_{i}-\mathbf{x}_{k}\rVert_{2}$ increases.  In the absence of proof, computational demonstrations in section \ref{sec:EntryDecay} illustrate that this is the case, at least for $\mathcal{L}=\Delta$ for $\mathbf{x}$ on the interior of $\Omega$.

\section{An Error Estimate and Existing Refinement Indicators} \label{sec:Error_Estimation}

To construct a method for approximating solutions to \eqref{eq:PDEexact} that automatically adapts to rapidly changing features in the solution, a quantity that reveals the need for refinement must be available.  This quantity will be referred to as a refinement indicator.  Further, it is useful if this quantity estimates the actual error in the solution, and in such a case the refinement indicator is called an error estimate.  To this end, notice that the results of section \ref{sec:PDESolution} could be analogously derived with $m$ replaced by $m+\mu$, with $\mu\in\mathbb{Z}$ and $\mu\geq1$. Defining now $\mathbf{u}_{m+\mu}=D_{m+\mu}^{-1}\mathbf{f}$, the difference between the two approximations yields
\begin{align}
    \mathbf{u}_{n,m+\mu}-\mathbf{u}_{n,m} = \mathbf{u}-\mathbf{u}_{n,m} +D_{m+\mu}^{-1}\mathbf{r}_{m+\mu}.\label{eq:errorestimator}
\end{align}
That is, the $k\textsuperscript{th}$ entry of the left side of \eqref{eq:errorestimator} provides an estimate of the error \eqref{eq:forwarderror}.

While \eqref{eq:errorestimator} provides an estimate of the error at each point in $\mathcal{S}$, there are other methods described in the literature that indicate where refinement of a node set should occur to improve the accuracy in the solution.  Two of these are described below.  In any case, these refinement indicators will be collected in a vector $\mathbf{u}_{\mbox{ref}}$ for which the $k\textsuperscript{th}$ entry will indicate the need for refinement in the node set around $\mathbf{x}_{k}$.  For instance, considering \eqref{eq:errorestimator} $\mathbf{u}_{\mbox{ref}}=\mathbf{u}_{\mbox{est}}\vcentcolon=\lvert\mathbf{u}_{n,m+\mu}-\mathbf{u}_{n,m}\rvert$.

Since refinement indicators are often intended to capture rapidly changing features in an unknown quantity, there have been several works that introduce an approximation of the gradient to identify such locations.  Utilizing the gradient as a refinement indicator poses some demonstrated problems when considering subsets of the domain where the solution has large curvature or is relatively flat \cite{OANH2017474}.  Therefore, the refinement indicator presented in \cite{OANH2017474} instead compares the local change in the solution at a subset of points (relative to the value of the solution at $\mathbf{x}_{k}$) to the change in a linear least squares fit of the solution at the same set of points (relative to the least squares fit evaluated at $\mathbf{x}_{k}$). This idea is adapted here to be compatible with the node refinement strategy discussed in algorithm \ref{alg:barynodeadd}. That is, the plane $\mathcal{P}_{k}(\mathbf{x})$ that minimizes
\begin{align}    \sum\limits_{j=1}^{n}\left(\mathbf{w}_{k,j}^{T}\mathbf{u}_{n,m}-\mathcal{P}_{k}(\mathbf{x}_{k,j})\right)^{2}\nonumber
\end{align}
is first determined by solving a linear least square problem with
\begin{align}
    [\mathbf{w}_{k,j}]_{l}=\left\{\begin{array}{cc}1,& \mathbf{x}_{k,j}=\mathbf{x}_{l}\\
    0, & \mbox{otherwise}\end{array}\right.,\nonumber
\end{align}
encoding the mapping between nodes in $\mathcal{S}$ and $\mathcal{N}_{k,n}$. The refinement indicator is then defined to have entries
\begin{align}
    [\mathbf{u}_{\mbox{ref}}]_{k} =[\mathbf{u}_{\mbox{ls}}]_{k}\vcentcolon=\max\limits_{j=1,2,\ldots,\tilde{n}}\left\lvert(\mathbf{w}_{k,j}^{T}\mathbf{u}_{n,m}-[\mathbf{u}_{n,m}]_{k})-(\mathcal{P}_{k}(\mathbf{x}_{k,j})-\mathcal{P}_{k}(\mathbf{x}_{k}))\right\rvert.\nonumber
\end{align}
for $\tilde{n}< n$.  In the experiments in this work $\tilde{n}=6$ has been chosen based on computational explorations and to best align with the idea presented in \cite{OANH2017474}.

Alternatively, other methods base refinement on the residual of a discrete approximation as in \cite{KosekSarler}.  In such a case, matrices $D_{n,m}$ and $D_{n,m+\mu}$ are both constructed, but only $\mathbf{u}_{n,m}$ is determined via the implicit solution of the system of equations $D_{n,m}\mathbf{u}_{n,m}=\mathbf{f}$.  Then the vector of refinement indicators is defined as
\begin{align}
\mathbf{u}_{\mbox{ref}}=\mathbf{u}_{\mbox{res}}\vcentcolon=\lvert D_{n,m+\mu}\mathbf{u}_{n,m}-\mathbf{f}\rvert.\nonumber
\end{align}
Since $D_{n,m+\mu}$ provides a more accurate approximation for the action of $\mathcal{L}$, this difference measures how well $\mathbf{u}_{n,m}$ satisfies the equations at higher order.

\section{An Implementation of Adaptive Local Kernel Approximations} \label{sec:Implementation}

The results developed in the previous sections warrant computational exploration both to demonstrate important aspects that are not shown (e.g., the decay of the entries of $(G_{n}D_{n,m})^{-1})$) and to compare the error estimate with existing refinement indicators.  These explorations are completed using the algorithms detailed in the following subsections.

\subsection{An Algorithm for Adaptive Kernel Based Approximation} \label{sec:Naive_Algorithm}

The adaptive algorithm utilized in this work is modified from the algorithm presented in \cite{JAR2024}.  It can be summarized concisely in three general steps:
\begin{enumerate}
\item given the current node set, $\mathcal{S}^{(l)}$, approximate the solution $\mathbf{u}_{n,m}^{(l)}$ (and, possibly, $\mathbf{u}_{n,m+\mu}^{(l)}$),
\item construct the vector of refinement indicators, $\mathbf{u}_{\mbox{ref}}^{(l)}$, and
\item add nodes to $\mathcal{S}$ near $\mathbf{x}_{k}$ for each $k$ where $\left\lvert[\mathbf{u}_{\mbox{ref}}^{(l)}]_{k}\right\rvert$ is too large.
\end{enumerate}
Note that here the superscript $(l)$ is introduced to indicate the number of steps taken in the outer loop of the adaptive algorithm \ref{alg:akba}, which provides more detail on each of these steps.
\begin{algorithm}\label{alg:akba}
    Adaptive kernel based approximation of solution to \eqref{eq:PDEexact} (modified from \cite{JAR2024})
\begin{algorithmic}[1]
    \State Given a set of $N^{(0)}$ nodes, $\mathcal{S}^{(0)}=\{\mathbf{x}_{k}^{(0)}\}_{k=1}^{N^{(0)}}$, the parameters $n$, $m$ and $\rho$ (and, if required, $\mu$) for constructing the local interpolants \eqref{eq:cardinterp}, a desired tolerance, $\varepsilon$, and a maximum number of refinement levels, $l_{\mbox{max}}$.
    \State Set $l=0$.
    \While{$l\leq l_{\mbox{max}}$}
        \For{$k\in\{1,2\ldots,N^{(l)}\}$}
        \State Determine the $n$ points in $S^{(l)}$ nearest to $\mathbf{x}_{k}^{(l)}$. Call this set of points $\mathcal{N}_{k,n}^{(l)}$.
        \If{$l=0$ OR there is an index $i$ such that $\mathbf{x}_{k}^{(l)}=\mathbf{x}_{i}^{(l-1)}$ and $\mathcal{N}_{k,n}^{(l)}\setminus\mathcal{N}_{i,n}^{(l-1)}\neq\emptyset$} \Comment{For $l=0$, approximations for the action of $\mathcal{L}$ will occur for every node in $\mathcal{S}^{(0)}$.  For $l>1$, the approximate action will be computed for nodes added at the previous step and recomputed for nodes that existed at the previous level for which the nearest neighbor set has changed.}
        \State Compute $(\mathcal{L}\psi_{k,n,m,j}^{(l)})(\mathbf{x}_{k}^{(l)})$ and, if necessary, $(\mathcal{L}\psi_{k,n,m+\mu,j}^{(l)})(\mathbf{x}_{k}^{(l)})$, $j=1,2,\ldots,n$ utilizing the nodes in $\mathcal{N}_{k,n}^{(l)}$. \Comment{Update the approximations where recomputation is necessary at this level.}
        \Else
        \State Set $(\mathcal{L}\psi_{k,n,m,j}^{(l)})(\mathbf{x}_{k}^{(l)})=(\mathcal{L}\psi_{i,n,m,j}^{(l-1)})(\mathbf{x}_{i}^{(l-1)})$ (and, if necessary, $(\mathcal{L}\psi_{k,n,m+\mu,j}^{(l)})(\mathbf{x}_{k}^{(l)})=(\mathcal{L}\psi_{i,n,m+\mu,j}^{(l-1)})(\mathbf{x}_{i}^{(l-1)})$) $j=1,2,\ldots,n$\Comment{Retain the existing approximation where recomputation is unnecessary.}
        \EndIf
        \EndFor
        \State Construct $D_{n,m}^{(l)}$ and solve $D_{n,m}^{(l)}\mathbf{u}_{n,m}^{(l)}=\mathbf{f}^{(l)}$ for $\mathbf{u}_{n,m}^{(l)}$.  If required, also construct $D_{n,m+\mu}^{(l)}$. Likewise, if required, solve $D_{n,m+\mu}^{(l)}\mathbf{u}_{n,m+\mu}^{(l)}=\mathbf{f}^{(l)}$ for $\mathbf{u}_{n,m+\mu}^{(l)}$
        \State Set $\mathcal{S}^{(l+1)}=\mathcal{S}^{(l)}$ \Comment{Initialize the node set for the next level of refinement.}
        \For{$k\in\{1,2\ldots,N^{(l)}\}$}
        \If{$\left\lvert[\mathbf{u}_{\mbox{ref}}^{(l)}]_{k}\right\lvert>\varepsilon$} \Comment{Determine if the refinement indicator is too large.}
        \State Add new nodes to the set $\mathcal{S}^{(l+1)}$ in a neighborhood of $\mathbf{x}_{k}^{(l)}$.
        \EndIf
        \EndFor
        \If{$\mathcal{S}^{(l+1)}\setminus\mathcal{S}^{(l)}=\emptyset$}
        \State Break from all loops. \Comment{No refinement occurred so stop the adaptive procedure.}
        \EndIf
    \EndWhile
    \State Set $L=l$ and return $\mathbf{u}_{n,m}^{(L)}$ as the approximation for $\mathbf{u}$. \Comment{Record the number of levels required to reach the desired tolerance and accept the current approximation of $\mathbf{u}$.}
\end{algorithmic}
\end{algorithm}

\subsection{Solution of the Large Systems of Linear Equations}

Line 10 of algorithm \ref{alg:akba} requires solving up to two independent systems of linear equations.  These can be handled in parallel, reducing the cost in terms of computation time to that of a single system of linear equations.  With $n\ll N^{(l)}$, these systems of equations are large (with $D_{n,m}^{(l)}$ and $D_{n,m+\mu}^{(l)}$ both $N^{(l)}\times N^{(l)}$) and sparse (with $nN^{(l)}$ nonzero entries).  Further, the results of section \ref{sec:CardinalApprox} indicate that the nonzero entries of row $k$ of these matrices have entries that are $O((h_{k,n}^{(l)})^{-\lvert\boldsymbol{\alpha}_{k,*}\rvert})$.  Algorithm \ref{alg:akba} only introduces new nodes where the refinement indicator is too large, so that as the number, $l$, of refinement levels increases, the values of $h_{k,n}^{(l)}$ may vary drastically dependent on $k$.  This can lead to poorly conditioned systems of linear equations.  To improve conditioning, $G_{n}^{(l)}D_{n,m}^{(l)}\mathbf{u}_{n,m}^{(l)}=G_{n}^{(l)}\mathbf{f}^{(l)}$ (likewise, $G_{n}^{(l)}D_{n,m+\mu}^{(l)}\mathbf{u}_{n,m+\mu}^{(l)}=G_{n}^{(l)}\mathbf{f}^{(l)}$, if necessary) is instead solved for $\mathbf{u}_{n,m}^{(l)}$.  To further improve stability, a sparse $LU$ factorization of $G_{n}^{(l)}D_{m,n}^{(l)}$ is constructed using row- and column-pivoting and scaling, implemented using Matlab's \verb+lu+ command.

\subsection{Node Refinement Via Local Delaunay Triangulations}

A desirable property of kernel based methods, like those considered here, is that a mesh defining connections between the nodes in $\mathcal{S}^{(l+1)}$ need not be maintained, rendering these methods ``meshless".  Still, certain strategies for constructing a mesh on a given set of points provide useful geometric information about the node set that can be leveraged to promote the satisfaction of condition 2 and to improve numerical stability.  For instance, the Delaunay triangulation of a point set in $\mathbb{R}^{2}$ defines a set of triangles that maximizes the smallest of the three angles of each triangle \cite{CompGeomDelaunay}.  This is useful when attempting to determine locations for nodes to add to $\mathcal{S}^{(l+1)}$ that are both near $\mathbf{x}_{k}$ and a sufficient distance from other nodes already in the set.  These properties are useful in an attempt to preserve local quasi-uniformity in the node set under refinement.

Where the need for refinement is indicated, computational experiments have illustrated that the addition of shifted barycenters (midpoints) of the triangles in a Delaunay triangulation is effective.  Specifically, when $\lvert[\mathbf{u}_{\mbox{ref}}^{(l)}]_{k}\rvert>\varepsilon$, for each triangle $t$ in the Delaunay triangulation of the node set $\mathcal{S}^{(l)}$ that has $\mathbf{x}_{k}$ as a vertex the midpoint of $t$ is perturbed by a small amount in a random direction and added to $\mathcal{S}^{(l+1)}$.  The addition of barycenters ensures that the new nodes are not too close to the vertices of the triangles, which are already in $\mathcal{S}^{(l+1)}$, and the favorable properties of the Delaunay triangulation promote the separation of the barycenters added near $\mathbf{x}_{k}$. Further, perturbation of the barycenters guards against issues in satisfying requirements of unisolvency of the node set indicated by condition 2.

The approach just described assumes that a mesh, the Delaunay triangulation of the entire node set, is constructed and maintained.  To avoid the need to maintain this mesh, algorithm \ref{alg:barynodeadd} constructs Delaunay triangulations $\mathcal{T}_{k,n'}^{(l)}$ on sets $\mathcal{N}_{k,n'}^{(l)}$ (with $n'\geq n$) of points in $\mathcal{S}^{(l)}$ nearest to $\mathbf{x}_{k}$ when $\lvert[\mathbf{u}_{\mbox{ref}}^{(l)}]_{k}\rvert>\varepsilon$.  Then perturbed barycenters of the triangles in $\mathcal{T}_{k,n'}^{(l)}$ with $\mathbf{x}_{k}$ as a vertex are added to $\mathcal{S}^{(l+1)}$ and the triangulation is discarded.  This is a process that is easily parallelized.  In section \ref{sec:Experiments} all computational experiments use $n'=100$.

Notice that this procedure of adding perturbed barycenters of triangles cannot introduce new nodes on the boundary of the domain.  Therefore, algorithm \ref{alg:barynodeadd} describes a process where, if a perturbed barycenter for a particular triangle is added to the node set and two of the triangle's vertices are on the boundary, then a perturbed midpoint of the edge whose vertices are on the boundary is also added to $\mathcal{S}^{(l+1)}$.  In some cases, these perturbed midpoints may not fall exactly on the boundary, specifically when the boundary is not a straight line.  So, an extra step of projecting the perturbed midpoint onto the boundary is also included in this algorithm.  Computational observations revealed that in certain cases this procedure added new nodes to the boundary too often, leading to numerical instabilities with nodes too close together.  To avoid this complication, the shifted midpoint is only added to $\mathcal{S}^{(l+1)}$ if the length of the edge is a large enough fraction of the average length of the edges of the triangle.

\begin{algorithm}\label{alg:barynodeadd}
    Addition of nodes as randomly shifted barycenters of a local Delaunay triangulation and boundary edges
    \begin{algorithmic}[1]
        \State Given a node $\mathbf{x}_k\in \mathcal{S}^{(l)}$ that is marked for refinement, a set $\mathcal{N}_{k,n'}^{(l)}$ of the $n'\geq n$ nearest neighbors to $\mathbf{x}_k$ in $\mathcal{S}^{(l)}$, $\vartheta\ll1$, and $\zeta\in(0,1]$.
        \State Let $\mathcal{T}_{k,n'}^{(l)}\vcentcolon=\{t_{j}\}_{j=1}^{J}$ be the Delaunay triangulation of $\mathcal{N}_{k,n'}^{(l)}$.\Comment{Construct the Delaunay triangulation.}
        \State Let $j_{1},j_{2},\ldots,j_{\eta}\subset\{1,2,\ldots,J\}$ be the indices of the triangles in $\mathcal{T}_{k,n'}^{(l)}$ for which $\mathbf{x}_{k}$ is a vertex.
            \For{$\nu=1,2,\ldots,\eta$}
                \State Set $r_1 = 1/3 + \vartheta \beta_{1}$, $r_2 = 1/3 + \vartheta \beta_{2}$ and $r_{3}=1-r_{1}-r_{2}$ with $\beta_{1}$ and $\beta_{2}$ chosen from the random uniform distribution on $(-1,1)$.
                \State Add $r_1 \mathbf{x}_{\nu,1} + r_2 \mathbf{x}_{\nu,2} + r_3 \mathbf{x}_{\nu,3}$ to $\mathcal{S}^{(l+1)}$, where $\mathbf{x}_{\nu,i}\in\mathcal{N}_{k,n'}^{(l)}$, $i=1,2,3$ are the coordinates of the vertices $t_{j_{\nu}}$
            \EndFor
        \State Let $\{\mathcal{E}_{\gamma}\}_{\gamma=1}^{\Gamma}$ be the set of distinct edges of all triangles in $\{t_{j_{\nu}}\}_{\nu=1}^{\eta}$ for which $\mathbf{x}_{k}$ is a vertex.  Denote the other vertex on each edge as $\mathbf{x}_{\gamma}$.
        \For{$\gamma=1,2,\ldots,\Gamma$}
            \If{$\mathbf{x}_{\gamma}\in\partial \Omega$}
                \If{there exists an edge in the set of distinct edges of all triangles in $\{t_{j_{\nu}}\}_{\nu=1}^{\eta}$ for which $\mathbf{x}_{\gamma}$ and the other vertex on the edge, denoted $\mathbf{x}_{\gamma'}$, is also both on the boundary and a vertex of an edge in $\{\mathcal{E}_{\gamma}\}_{\gamma=1}^{\Gamma}$}
                    \If{$\lVert\mathbf{x}_{\gamma}-\mathbf{x}_{\gamma'}\rVert_{2}\geq\zeta\left(\lVert\mathbf{x}_{\gamma}-\mathbf{x}_{\gamma'}\rVert_{2}+\lVert\mathbf{x}_{\gamma}-\mathbf{x}_{\gamma''}\rVert_{2}+\lVert\mathbf{x}_{\gamma'}-\mathbf{x}_{\gamma''}\rVert_{2}\right)$, with $\mathbf{x}_{\gamma''}$ the third vertex of the triangle for which $\mathbf{x}_{\gamma}$ and $\mathbf{x}_{\gamma''}$ are the vertex of an edge.\Comment{Avoid adding nodes to edges that are already much smaller than those nearby. }}
                        \State Compute $\mathbf{x}_{\gamma}+(1/2+\vartheta\beta)(\mathbf{x}_{\gamma'}-\mathbf{x}_{\gamma})$, with $\beta$ chosen from the random uniform distribution on $(-1,1)$, and project it to the closest point on the boundary.  Add this projected point to $\mathcal{S}^{(l+1)}$. \Comment{On boundaries that are not straight lines, the shifted midpoints will no longer be on the boundary.  So, project the shifted midpoint to the boundary before adding the point to $\mathcal{S}^{(l+1)}$.}
                    \EndIf
                \EndIf
            \EndIf
        \EndFor
   \end{algorithmic}
\end{algorithm}

Delaunay triangulations of node sets are unique, except when it is possible for more than three nodes to fall on the same circumcircle of a triangle in the triangulation \cite{CompGeomDelaunay}. This implies that the same triangle is likely appear in many different sets $\mathcal{T}_{k,n'}^{(l)}$. To take advantage of parallel computing, a perturbed barycenter (and, possibly, one or more perturbed edge midpoints) of a particular triangle is added to $\mathcal{S}^{(l+1)}$ for each of its vertices indicated for refinement.  So, after completion of lines 12 through 14 of \ref{alg:akba}, there are likely nodes in $\mathcal{S}^{(l+1)}$ that are too close together.  To remedy this unfortunate consequence of using parallel computing, nodes can be removed from the set through the use of clustering algorithms and choosing one node for each cluster.  The implementation used to generate the results in section \ref{sec:Experiments} takes advantage of Matlab's \verb+uniquetol+ command with a tolerance larger than the parameter specifying the size of the random perturbation in algorithm \ref{alg:barynodeadd} yet smaller than the expected minimum expected distance between nodes when algorithm \ref{alg:akba} terminates.

\section{Computational Experiments} \label{sec:Experiments}

The algorithms described in section \ref{sec:Implementation} are applied to a set of prototypical examples in the following subsections through a series of computational experiments.  These examples are designed to highlight the utility of the error estimator described in section \ref{sec:Error_Estimation} as a refinement indicator when compared to the other existing indicators described in that section.

\subsection{Test Functions, Domains and Parameter Selections for Computational Experiments}

The performance of algorithm \ref{alg:akba} utilizing algorithm \ref{alg:barynodeadd} at line 14 is demonstrated on four separate test functions chosen to be solutions of \eqref{eq:PDEexact} when $\mathcal{L}=\Delta$ for $\mathbf{x}$ on the interior of $\Omega$.  The first three test functions are considered on the domain $\Omega_{1}=\{\mathbf{x}\in\mathbb{R}^{2}:\lVert\mathbf{x}-0.5\cdot\mathbf{1}\rVert_{\infty}\leq0.5\}$.  This domain is the unit square centered at $(0.5,0.5)$ and it has an area of 1. On this domain $\mathcal{L}$ is the identity map for $\mathbf{x}$ on the boundary to impose Dirichlet boundary conditions.  Further, the test functions
\begin{align}
    u_{1}(\mathbf{x}) = \sum\limits_{i=1}^{2d}\frac{1}{1+a\left\lVert\mathbf{x}-\mathbf{y}_{i}\right\rVert_{2}^{2}}\nonumber
\end{align}
and
\begin{align}
    u_{2}(\mathbf{x}) = \sum\limits_{i=1}^{2d}\exp(-a\left\lVert\mathbf{x}-\mathbf{y}_{i}\right\rVert_{2}^{2}),\nonumber
\end{align}
both with $\mathbf{y}_{1}=\left[\begin{array}{cc}0.66  & 0.89\end{array}\right]^{T}$, $\mathbf{y}_{2}=\left[\begin{array}{cc}0.74 & 0.02\end{array}\right]^{T}$, $\mathbf{y}_{3}=\left[\begin{array}{cc}0.09 &  0.86\end{array}\right]^{T}$, $\mathbf{y}_{4}=\left[\begin{array}{cc}0.24 & 0.08\end{array}\right]^{T}$, and
\begin{align}
    u_{3}(\mathbf{x}) = \ln(a\lVert\mathbf{x}+0.01\cdot\mathbf{1}\rVert_{2}^{2})+1\nonumber
\end{align}
are each considered in turn.  Graphically, $u_{1}$ and $u_{2}$ are similar with features that become more localized as $a$ increases.  However, \cite{JAR2024} highlights that $u_{1}$
can be expressed as a power series with radius of convergence $\left\lVert\mathbf{x}-\mathbf{y}\right\rVert_{2}<1/\sqrt{a}$ and terms that grow for all $j$ when outside the radius of convergence, while the power series for $u_{2}$ has infinite radius of convergence and terms that grow only until finite $j$ (for each fixed $\left\lVert\mathbf{x}-\mathbf{y}\right\rVert_{2}$).  On the other hand, $u_{3}$ features a singularity just outside the domain.  In all cases shown here, $a=1000$.

The fourth test integrand is
\begin{align}
    u_{4}(\mathbf{x}) = \lVert\mathbf{x}-\mathbf{y}_{0}\rVert_{2}^{\varrho}\sin\left(\frac{1}{2}\theta(\mathbf{x}-\mathbf{y}_{0})\right), \nonumber
\end{align}
where
\begin{align}
    \theta(\mathbf{x}) = \tan^{-1}\left(\frac{\mathbf{e}_{2}^{T}R\mathbf{x}}{\mathbf{e}_{1}^{T}R\mathbf{x}}\right) \nonumber
\end{align}
with $\mathbf{e}_{j}$ the $j\textsuperscript{th}$ column of the $2\times2$ identity matrix and
\begin{align}
    R = \sqrt{2}/2\left[\begin{array}{cc}-1 & 1\\-1 &-1\end{array}\right].\nonumber
\end{align}
This test integrand features a singularity at $\mathbf{y}_{0}$ in partial derivatives $\partial^{\boldsymbol{\alpha}}$ for which $\varrho/2<\lvert\boldsymbol{\alpha}\rvert$, which violates the assumptions required to express the interpolation error as \eqref{eq:interperror} when $m+1>\varrho/2$.  The domain for the fourth test function is
\begin{align}
    \Omega_{2} = \left\{\mathbf{x}\in\mathbb{R}^{2}\vcentcolon \max\left(\lVert\mathbf{x}\rVert_{2}-\sqrt{\frac{4}{3\pi}},\sqrt{\frac{4}{3\pi}}-\left\lVert\mathbf{x}-\sqrt{\frac{4}{3\pi}}\left[\begin{array}{c}1 \\ -1\end{array}\right]\right\rVert_{\infty}\right)\leq0\right\},\nonumber
\end{align}
which is the circle of radius $\sqrt{4/(3\pi)}$ with the fourth quadrant deleted.  This domain also has an area of 1, but in this case $\mathcal{L}$ is the Laplace operator on the interior of $\Omega_{2}$,
\begin{itemize}
    \item the directional derivative in the direction of the outward pointing normal vector for $\mathbf{x}\in\Omega$ with $\lVert\mathbf{x}\rVert_{2}-\sqrt{4/(3\pi)}=0$ (i.e., on the circular portion of the boundary) and
    \item the identity map to enforce Dirichlet boundary conditions on the remainder of the boundary.
\end{itemize}
Further, the node refinement algorithm \ref{alg:barynodeadd} is modified to ensure that the sets of nearest neighbors, $\mathcal{N}_{k,n}$, are selected to contain the point of singularity in the derivatives either on the boundary of or outside of the convex hull of the set.  This is implemented by choosing for $\mathcal{N}_{k,n}$ the $n$ points in $\mathcal{S}$ closest to $\mathbf{x}_{k}$ that also satisfy $\mathbf{x}_{k,j}^{T}\mathbf{x}_{k}\geq0$.  That is, the points in $\mathcal{N}_{k,n}$ are on the same side as $\mathbf{x}_{k}$ of the line through the origin that is perpendicular to $\mathbf{x}_{k}$.

In all cases presented here, $\varphi(r)=r^3$ is used as the kernel.  The degree, $m$, of the supplemental polynomial basis ranges from $2$ to $4$, while $\mu$ takes on the values $1$ or $2$ for those refinement indicators that require it.  The number of nearest neighbors in the set $\mathcal{N}_{k,n}$ is set to $n=2M_{d,m}$, a choice that has been demonstrated to mitigate loss of accuracy near the boundary (see, e.g., \cite{JARBF2017}).  Experiments also consider eleven different tolerances ($\varepsilon$) on the refinement indicators, ranging from $10^{-4}$ to $10^{-2}$.  The logarithms of these tolerances are equally spaced.  For each combination of $m$, $\mu$, and $\varepsilon$, the adaptive algorithm \ref{alg:akba} is run ten separate times with the set $\mathcal{S}^{(0)}$ randomly perturbed from the sets of points illustrated in figure \ref{fig:InitialNodeSets}.  Except for those appearing at corners of these domains, each these nodes is perturbed by a (Euclidean) distance of $10^{-2}$ in a random direction, while ensuring that a node that starts on the boundary remains on the boundary.  Computation times, total numbers of nodes, and a metric indicating how well each refinement indicator predicts the actual forward error in the solution are collected at the conclusion of algorithm \ref{alg:akba} for each combination of $m$, $\mu$, and $\varepsilon$ and each initialization of the node sets. The characteristic spacing, $h_{k,n}^{(l)}$ is taken to be the minimum distance from $\mathbf{x}_{k}$ to any other point in $\mathcal{N}_{k,n}^{(l)}$ when solving \eqref{eq:weightrelationshipscaled} and constructing $G_{n}^{(l)}$.  Also, within algorithm \ref{alg:barynodeadd}, $\vartheta=10^{-3}$ and $\zeta=0.25$. Note that all computations presented here were performed on a workstation with two Intel$^\circledR$ Xeon$^\circledR$ CPU E5-2697 v3 processors, each running at 2.60GHz, and 256 GB of memory running MATLAB R2022b.

\begin{figure}
    \centering
    \includegraphics[width=0.5\linewidth]{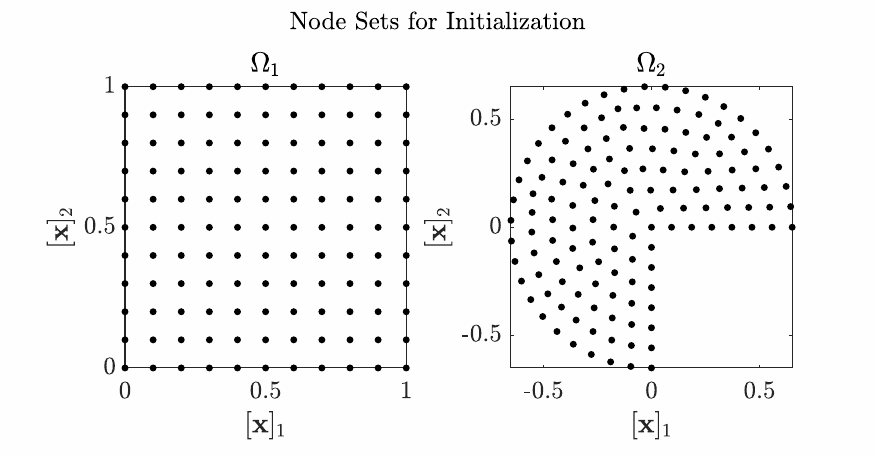}
    \caption{Initial node distributions used for the two-dimensional computational demonstrations on $\Omega_{1}$ (left) and $\Omega_{2}$ (right).}
    \label{fig:InitialNodeSets}
\end{figure}

\subsubsection{Demonstration of Decay in the Entries of $(G_{n}D_{n,m})^{-1}$} \label{sec:EntryDecay}

The $k\textsuperscript{th}$ row of $(G_{n}D_{n,m})^{-1}$ can be computed by solving a system of linear equations with coefficient matrix $(G_{n}^{(l)}D_{n,m}^{(l)})^{T}$ and the $k\textsuperscript{th}$ column of the identity matrix as the right side. For each row, in turn, a function that interpolates the set of ordered pairs
\begin{align}
    \left\{\left(\lVert\mathbf{x}_{i}-\mathbf{x}_{k}\rVert_{2},\frac{\lvert[(G_{n}^{(l)}D_{n,m}^{(l)})^{-1}]_{ki}\rvert}{\max\limits_{j=1,2,\ldots,N^{(l)}}\lvert[(G_{n}^{(l)}D_{n,m}^{(l)})^{-1}]_{kj}\rvert}\right)\right\}_{i=1}^{N}\nonumber
\end{align}
can be constructed and evaluated at a set of fixed distances $\{\chi_{j}\}_{j=1}^{n'}$ (independent of $k$). Denoting this interpolating function for the $k\textsuperscript{th}$ row as $\omega_{k}^{(l)}(\chi)$, with $\chi$ representing distance, figure \ref{fig:MinvDecay} illustrates that $\lvert[(G_{n}^{(l)}D_{n,m}^{(l)})^{-1}]_{ki}\rvert$ decreases exponentially as $\chi$ increases.  The dashed curves pass through the points of $(\chi_{j},\max\limits_{k=1,2,\ldots,N^{(l)}}\omega_{k}^{(l)}(\chi_{j}))$ at each level of refinement, while the solid curves demonstrate the weighted average
\begin{align}
    \frac{\sum\limits_{k=1}^{N^{(l)}}h_{k,n}^{(l)}\omega_{k}^{(l)}(\chi_{j})}{\sum\limits_{k=1}^{N^{(l)}}h_{k,n}^{(l)}}
\end{align}
at each distance, $\chi_{j}$.  This exponential decrease in the size of $(G_{n}^{(l)}D_{n,m}^{(l)})^{-1}$ with distance is precisely what is desired for an algorithm to be able to realize improvements in accuracy through local changes in spacing as indicated in section \ref{sec:MatInverse}.

\begin{figure}
    \centering
    \includegraphics[width=\linewidth]{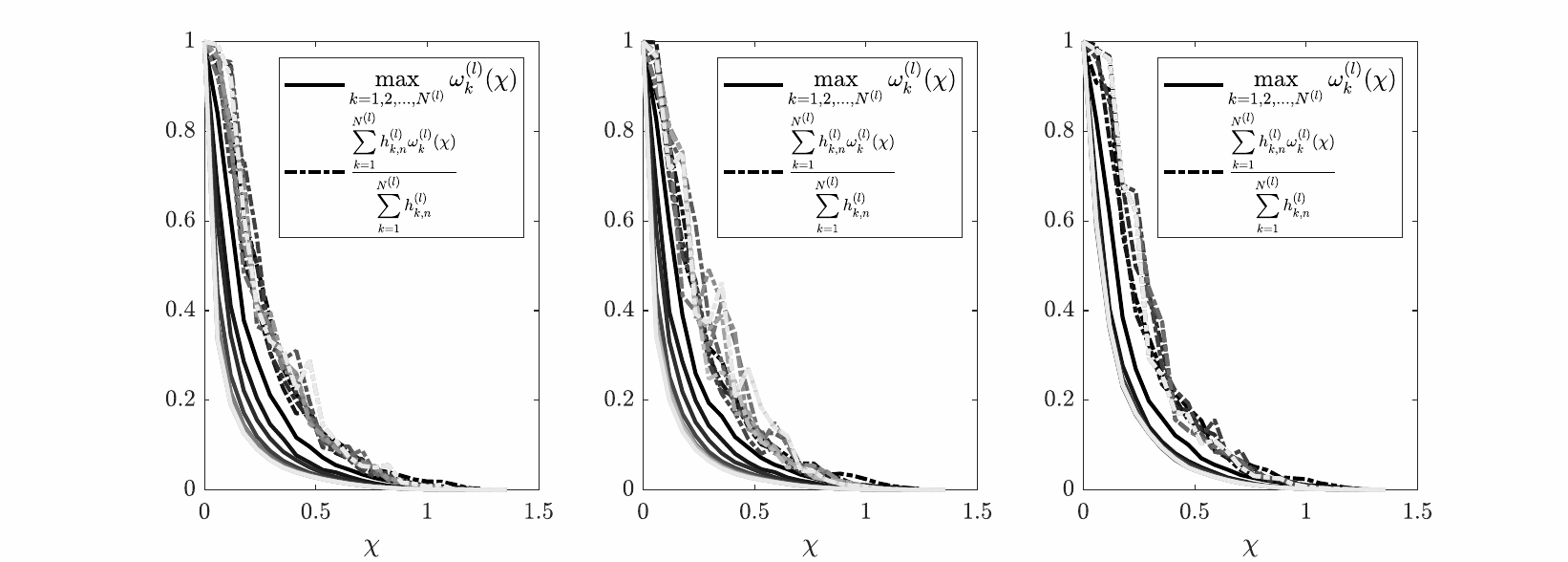}
    \caption{Computational evidence of exponential decay of normalized entries in the rows of $(G_{n}^{(l)}D_{n,m}^{(l)})^{-1}$ with $\rho=3$, $m=1$ and $\mu=2$.  Columns correspond to test function $u_{1}$, $u_{2}$ and $u_{3}$, respectively.  Darker shades of curves indicate lower levels, $l$, of the adaptive procedure.  }
    \label{fig:MinvDecay}
\end{figure}

\subsection{Comparisons With Existing Refinement Indicators}

First, figure \ref{fig:AdaptiveSpacing} illustrates that each refinement indicator (defined in section \ref{sec:Error_Estimation}), when used within algorithm \ref{alg:akba}, promotes a decrease in the characteristic spacing $h_{k,n}$ near rapidly changing features of the solution.  In the figure, log base 10 of the characteristic spacing $h_{k,n}$ at the conclusion of the adaptive algorithm \ref{alg:akba} is shown when applied to solving \eqref{eq:PDEexact} with solution $u_{1}$.  In this case the characteristic spacing is the minimum of the distance from $\mathbf{x}_{k}$ to any point in $\mathcal{N}_{k,n}^{(L)}$.
\begin{figure}
    \centering
    \includegraphics[width=\linewidth]{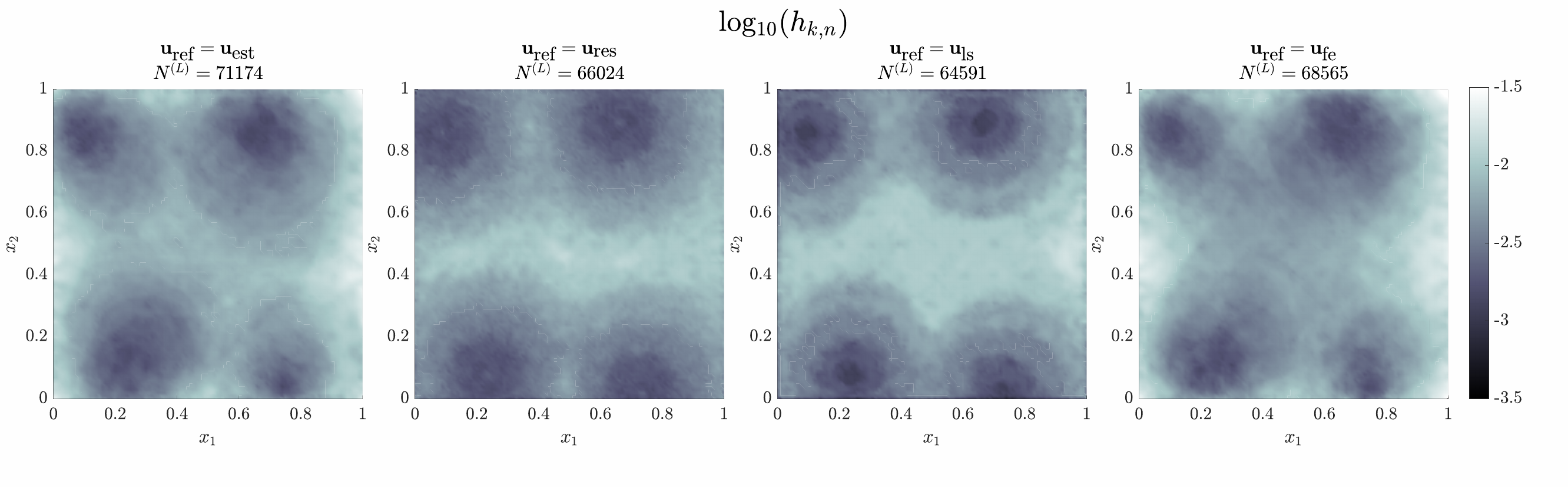}
    \caption{Log base 10 of the characteristic spacing $h_{k,n}$ at the conclusion of the adaptive algorithm \ref{alg:akba} when applied to solving \eqref{eq:PDEexact} with solution $u_{1}$ with $\rho=3$, $m=2$ and $\mu=1$.  Here the characteristic spacing is the minimum of the distance from $\mathbf{x}_{k}$ to any point in $\mathcal{N}_{k,n}^{(L)}$.  Notice that for each choice of $\mathbf{u}_{\mbox{ref}}$, the algorithm decreases spacing near the localized features of the solution, as desired.  To achieve values of $N^{(L)}$ that are nearly the same across refinement indicators, tolerances of $\varepsilon=10^{-3}$ ($\mathbf{u}_{\mbox{ref}}=\mathbf{u}_{\mbox{est}}$), $\varepsilon=10^{-4.8}$ ($\mathbf{u}_{\mbox{ref}}=\mathbf{u}_{\mbox{res}}$), $\varepsilon=10^{-3.25}$ ($\mathbf{u}_{\mbox{ref}}=\mathbf{u}_{\mbox{ls}}$) and $\varepsilon=10^{-3}$ ($\mathbf{u}_{\mbox{ref}}=\mathbf{u}_{\mbox{fe}}$) were chosen.}
    \label{fig:AdaptiveSpacing}
\end{figure}

Next, figures \ref{fig:SquareNumberofNodes} to \ref{fig:CircleNumberofNodes} demonstrate the performance of the error estimator analyzed in sections \ref{sec:ErrorApprox} and \ref{sec:Error_Estimation} when used as a refinement indicator.  This performance is illustrated in comparison to results generated utilizing the other refinement indicators described in section \ref{sec:Error_Estimation} within the algorithm \ref{alg:akba}.  Further, the actual forward error is introduced as an ideal refinement indicator (i.e., $\mathbf{u}_{\mbox{ref}}=\mathbf{u}_{\mbox{fe}}\vcentcolon=\lvert\mathbf{u}-\mathbf{u}_{n,m}\rvert$) since the intent is to utilize algorithm \ref{alg:akba} to determine a solution to \eqref{eq:PDEexact} with forward error less than a specified tolerance $\varepsilon$ through the node refinement strategy in algorithm \ref{alg:barynodeadd} where the indicator is too large.  In each of these figures, the horizontal axis represents the base ten logarithm of the maximum forward error in the solution at the conclusion of the algorithm \ref{alg:akba} for a set of eleven tolerances ($\varepsilon$) ranging from $10^{-4}$ to $10^{-2}$ for which their base ten logarithms are equally spaced.  That is, each tolerance is a value $\varepsilon$ for which $\log_{10}(\varepsilon)=-4+k/5$ for some $k\in\{0,1,\ldots,10\}$. Further, in each figure the colors of the curves (where available) indicate the refinement indicator, the markers signify different values of $m$, and the line styles distinguish different values of $\mu$ for those indicators where the choice of $\mu$ is relevant.  In all cases, the curves represent the average of the displayed quantity over the ten random initilizations of $\mathcal{S}^{(0)}$ for each tolerance.

For each of the test functions $u_{j}$, $j=1,2,3$, figure \ref{fig:SquareNumberofNodes} demonstrates that the number of nodes, $N$, required to achieve a particular maximum forward error increases as the error decreases.  Further, in agreement with the dependence of \eqref{eq:forwarderror} on the polynomial order, fewer nodes are required to achieve a particular level of error as $m$ increases.  Notably, each refinement indicator produces similar results with respect to the number of required nodes, which indicates some consistency for each test function relative to the number of nodes required to resolve its features.

\begin{figure}
    \centering
    \includegraphics[width=\linewidth]{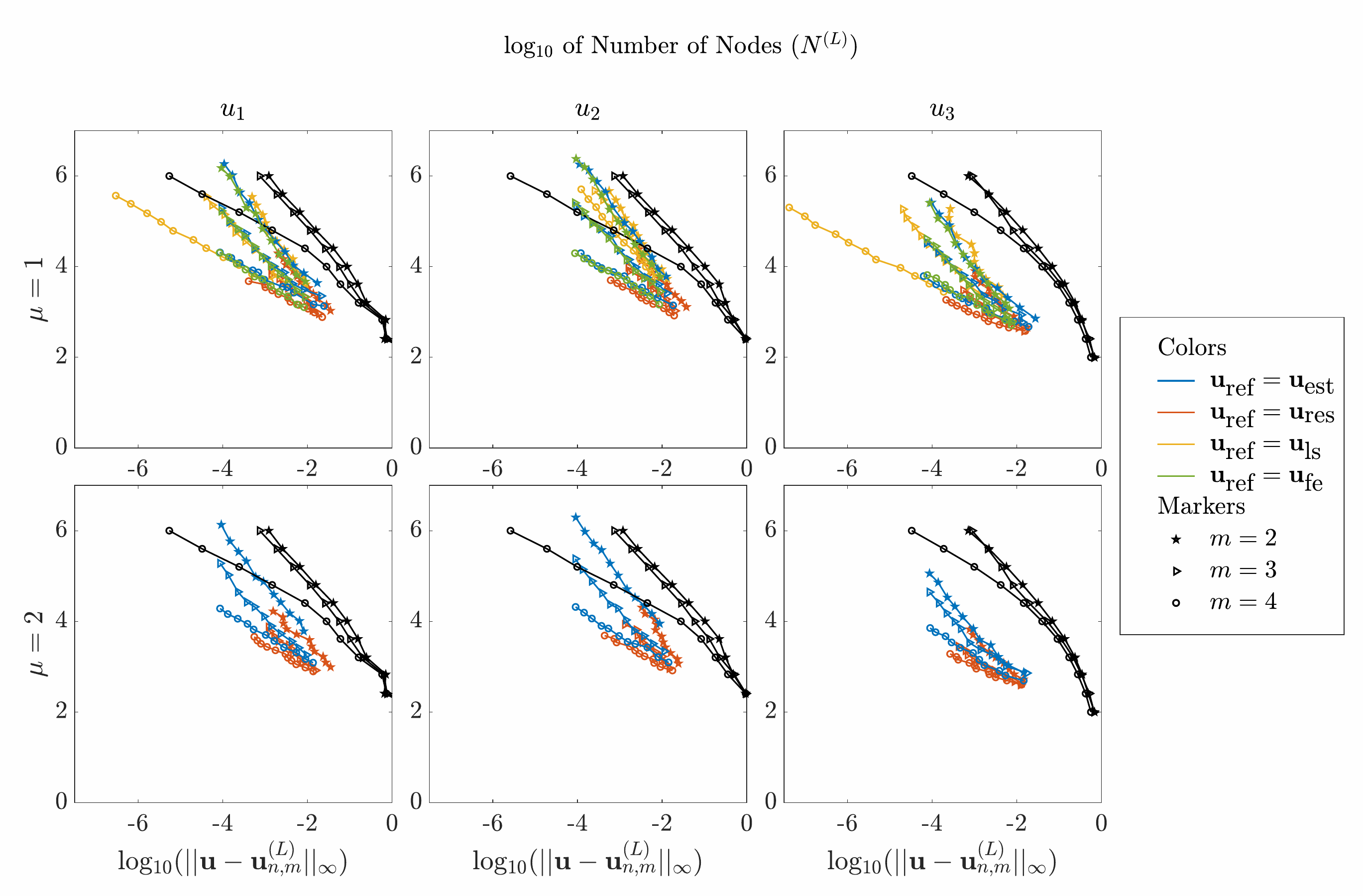}
    \caption{Base 10 logarithm of number of nodes, $N^{(L)}$, required to compute the solution to \eqref{eq:PDEexact} with base ten logarithm of the forward error indicated by the horizontal axis for $u_{j}$, $j=1,2,3$.  Notice that the number of nodes required is consistent with respect to the refinement indicator. In all frames, the black curves illustrate the relationship between the error and number of nodes when $N$ quasi-uniformly spaced nodes are used to solve the equation (i.e., there is no refinement).  For each fixed choice of $m$ and $\mu$ that is shown, the adaptive methods all outperform uniform refinement over this range of errors.}
    \label{fig:SquareNumberofNodes}
\end{figure}

The results illustrated in figure \ref{fig:SquareNumberofNodes}, and all subsequent figures, are generated by employing algorithm \ref{alg:akba} with the same set of tolerances.  Considering a single tolerance, e.g., $10^{-4}$ (indicated by the farthest left point on each curve in the figure), the maximum forward error in the solution at the conclusion of the adaptive algorithm can vary widely from one refinement indicator to another (except between when comparing $\mathbf{u}_{\mbox{est}}$ to $\mathbf{u}_{\mbox{fe}})$.  This can be explained by the results presented in figure \ref{fig:SquareErrorRatios}, which depicts
\begin{align}
   \max\limits_{k=1,2,\ldots,N}\frac{\left\lvert\left[\mathbf{u}-\mathbf{u}_{n,m}^{(L)}\right]_{k}\right\rvert}{\max\left\{\varepsilon,\left[\mathbf{u}_{\mbox{ref}}^{(L)}\right]_{k}\right\}},\label{eq:errorratio}
\end{align}
a quantity denoted the error ratio.  The error ratio considers, point-wise, the relative size of the forward error to the maximum of either the error estimate or the specified tolerance.  For a refinement indicator that also serves as an error estimator, this quantity should approach the value of one (or its base 10 logarithm 0), since a reliable estimator should provide a better approximation for the forward error as $h_{k,n}$ increases.  The indicator $\mathbf{u}_{\mbox{est}}$ is designed to perform this way, with \eqref{eq:errorestimator} indicating that the difference between the forward error and the estimate is exactly the forward error in the solution based on the inclusion of a polynomial basis of order $m+\mu$ instead of $m$.  Therefore, in figure \ref{fig:SquareErrorRatios} the base ten logarithm of the error ratio for $\mathbf{u}_{\mbox{est}}$ (and, of course, $\mathbf{u}_{\mbox{fe}}$) approaches $0$ as the maximum forward error decreases.  This, however, is not the case for $\mathbf{u}_{\mbox{ls}}$ or $\mathbf{u}_{\mbox{res}}$, since these indicators are not designed to estimate the forward error. In fact, the base ten logarithm of the error ratio diverges from the value of 0 for these two indicators.

\begin{figure}
    \centering
    \includegraphics[width=\linewidth]{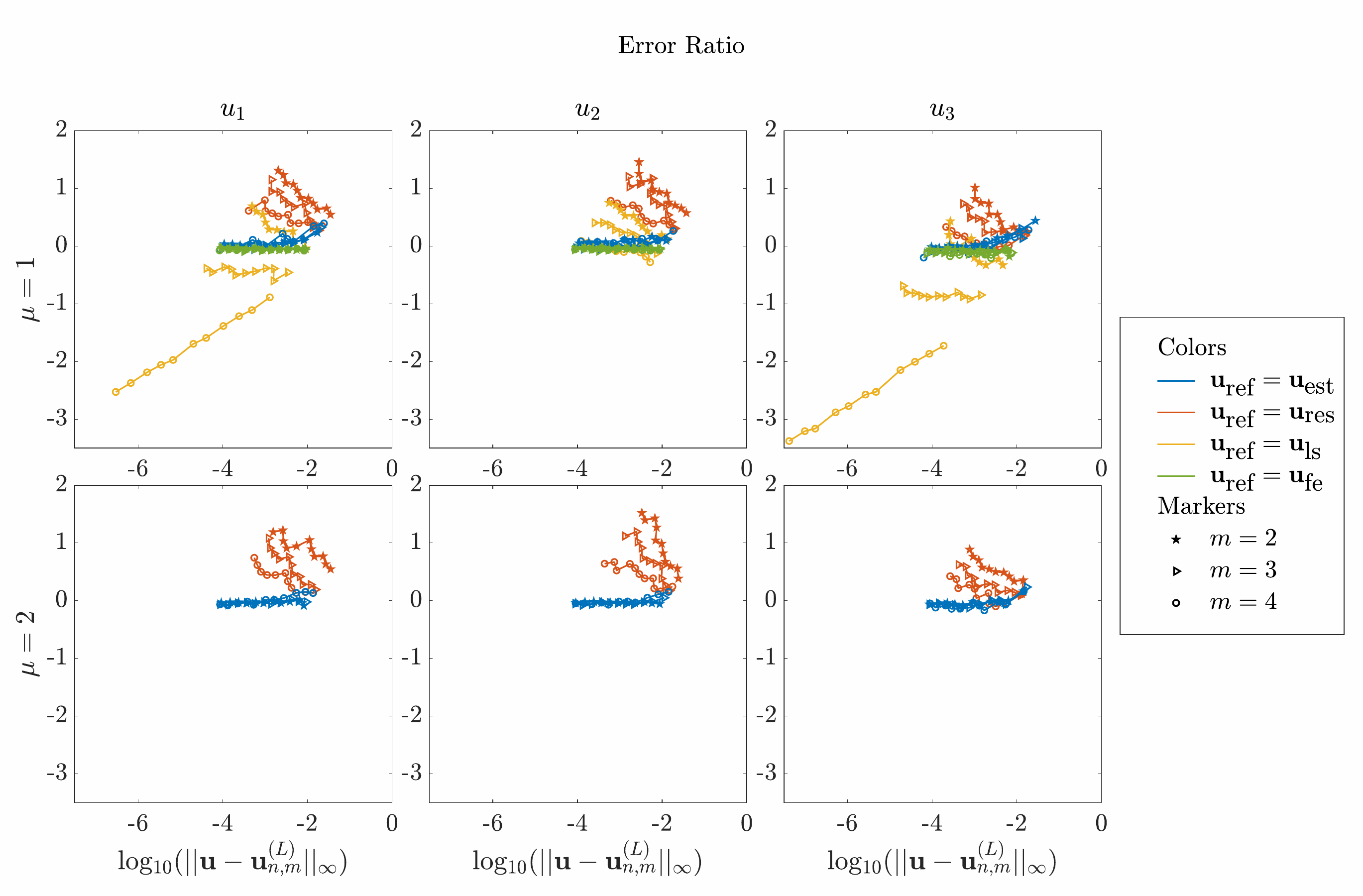}
    \caption{Base ten logarithm of the error ratio \eqref{eq:errorratio} when computing the solution to \eqref{eq:PDEexact} with base ten logarithm of the forward error indicated by the horizontal axis for $u_{j}$, $j=1,2,3$. Notice that for $\mathbf{u}_{\mbox{est}}$ and $\mathbf{u}_{\mbox{fe}}$ the error ratio approaches 1, i.e., the base ten logarithm approaches zero, as the forward error decreases.  This is not the case for either $\mathbf{u}_{\mbox{ls}}$ or $\mathbf{u}_{\mbox{res}}$.}
    \label{fig:SquareErrorRatios}
\end{figure}

The inability to estimate the forward error may not be a problem if a user of an adaptive algorithm only wishes to attempt to capture localized features of a solution to a PDE. However, proper selection of the tolerance, $\varepsilon$, to ensure the resolution of these features may be more difficult. That is, the deviation of the error ratio from the desired value of 1 can vary greatly depending on the function being approximated.  Even with $u_{1}$ and $u_{2}$ being similar visually, there is a clear difference when comparing the deviation of the base ten logarithm of the error ratio from zero in the left two subplots of figure \ref{fig:SquareErrorRatios}.  This might be explainable by the convergence characteristics of the power series of $u_{1}$ and $u_{2}$.  That is, the infinite radius of convergence of the power series of $u_{1}$ may explain the ability of $\mathbf{u}_{\mbox{ls}}$ to better capture the forward error than in the case of $u_{1}$ with a finite radius of convergence.  However, a thorough investigation of this idea has not been conducted.

Due to the ability to solve the two large systems of equations in line 10 of the algorithm \ref{alg:akba} in parallel, the computational cost (with respect to wall clock time) when using each of $\mathbf{u}_{\mbox{est}}$ and $\mathbf{u}_{\mbox{res}}$ is similar, as illustrated in figure \ref{fig:SquareComputationTimes}. For these two refinement indicators, the cost is also consistent when considering each of the test functions, in turn.  On the other hand, the cost of the algorithm when utilizing $\mathbf{u}_{\mbox{ls}}$ as the estimator is often greater than for the other two estimators.

\begin{figure}
    \centering
    \includegraphics[width=\linewidth]{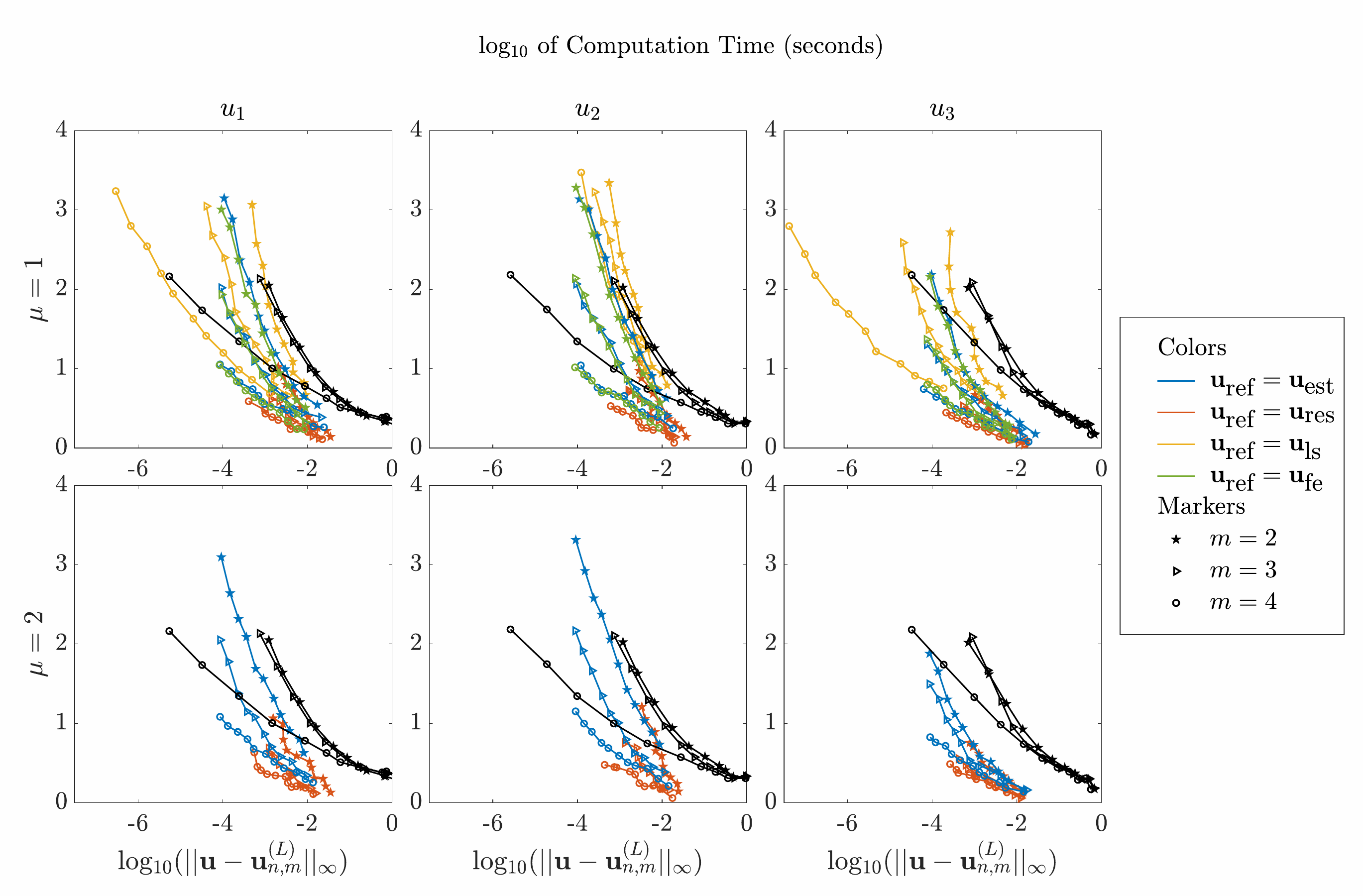}
    \caption{Base ten logarithm of computation times in seconds when computing the solution to \eqref{eq:PDEexact} with base ten logarithm of the forward error indicated by the horizontal axis for $u_{j}$, $j=1,2,3$.  The refinement indicator $\mathbf{u}_{\mbox{ls}}$ incurs greater computational cost as a consequence of the increased number of refinement levels required to achieve the prescribed tolerance as illustrated in figure \ref{fig:SquareLevels}.  In all frames, the black curves illustrate the relationship between the error and computation times when quasi-uniformly spaced nodes are used to solve the equation (i.e., there is no refinement).  For each fixed choice of $m$ and $\mu$ that is shown, the adaptive methods all outperform uniform refinement over this range of errors.}
    \label{fig:SquareComputationTimes}
\end{figure}

The increased computational cost when utilizing $\mathbf{u}_{\mbox{ls}}$ is due to an increased number of refinement levels required to achieve the prescribed tolerance as illustrated in figure \ref{fig:SquareLevels}.  Further, figures \ref{fig:SquareComputationTimes} and \ref{fig:SquareLevels} both demonstrate that for $\mathbf{u}_{\mbox{ls}}$ as the error indicator, the computational cost driven by the number of refinement levels can vary by orders of magnitude depending on the test function being considered.  Note that, in the demonstrations provided here, $l_{\mbox{max}}=100$ was never reached and the algorithm terminated when the refinement indicator met the desired tolerance at every node.

\begin{figure}
    \centering
    \includegraphics[width=\linewidth]{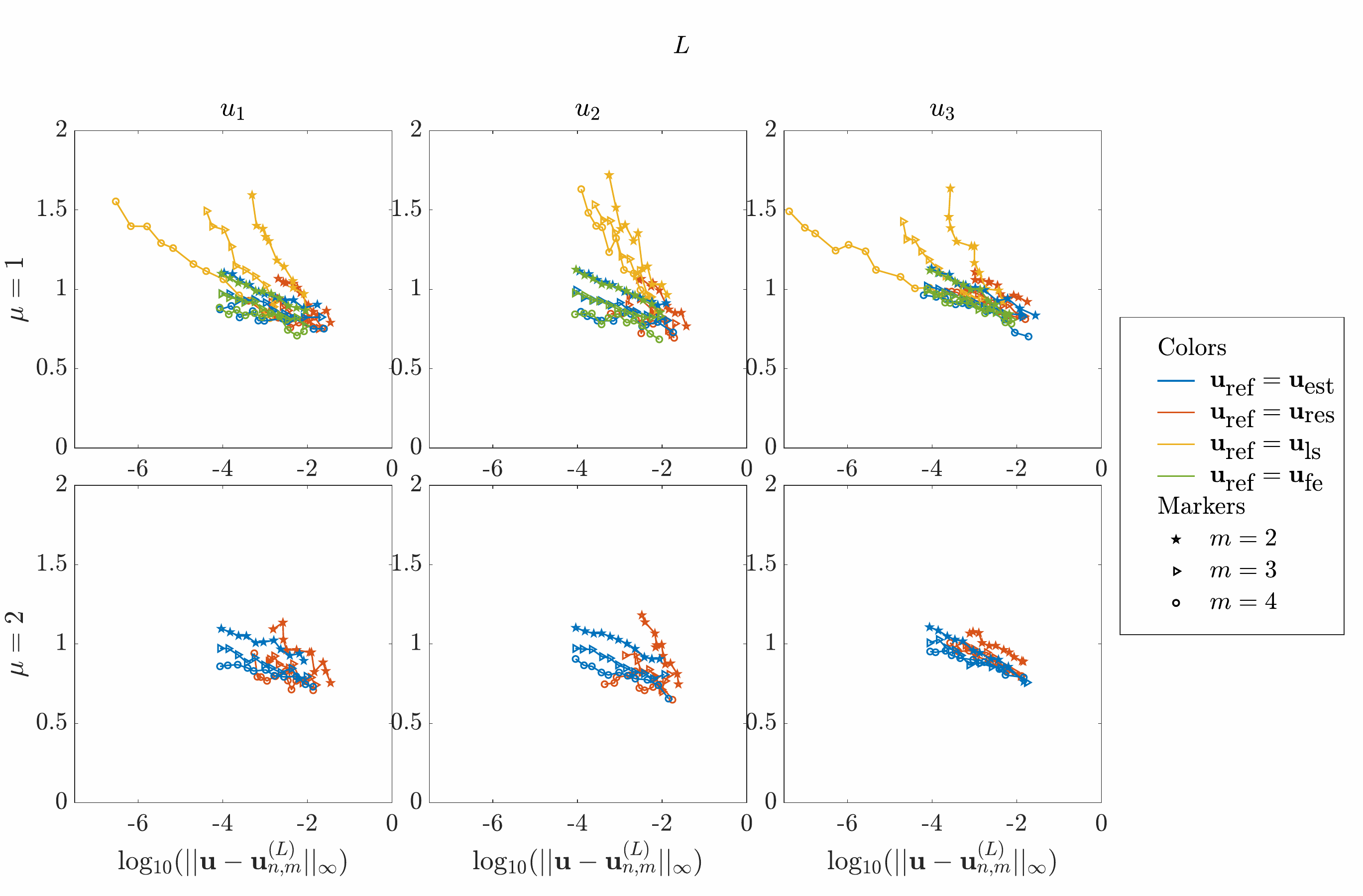}
    \caption{Base ten logarithm of the number of levels required to achieve the prescribed tolerance on the refinement indicator when computing the solution to \eqref{eq:PDEexact} with base ten logarithm of the forward error indicated by the horizontal axis for $u_{j}$, $j=1,2,3$.  The refinement indicator $\mathbf{u}_{\mbox{ls}}$ often requires a larger number of refinement levels.}
    \label{fig:SquareLevels}
\end{figure}

Beyond these demonstrations, to illustrate the dependence of the performance of the refinement indicators investigated here on the assumptions made to express the forward error as
\eqref{eq:forwarderror} (namely, the continuity of partial derivatives of $u$ to a particular order), figure \ref{fig:CircleNumberofNodes} provides evidence of the importance of these assumptions.  In particular, when $\varrho=1$ and $\mathbf{y}_{0}=\mathbf{0}$ (left frame of the figure) the test function $u_{4}$ has a discontinuity at the origin in its first derivative.  This is a violation of the assumptions required for the existence of the Taylor series expansion for all choices of $m$ and $m+\mu$ considered here.  Therefore, the number of nodes $N$ required to achieve a particular forward error is similar for each error indicator and regardless of $m$ and $\mu$.  However, if the singularity is moved outside the domain, as with middle frame of the figure, or if the singularity is in the second derivative instead of the first, as in the right-most frame, increases in $m$ can reduce the number of nodes required to achieve a particular maximum forward error since the assumptions are no longer violated for certain values of $m$ and $m+\mu$.

\begin{figure}
    \centering
    \includegraphics[width=\linewidth]{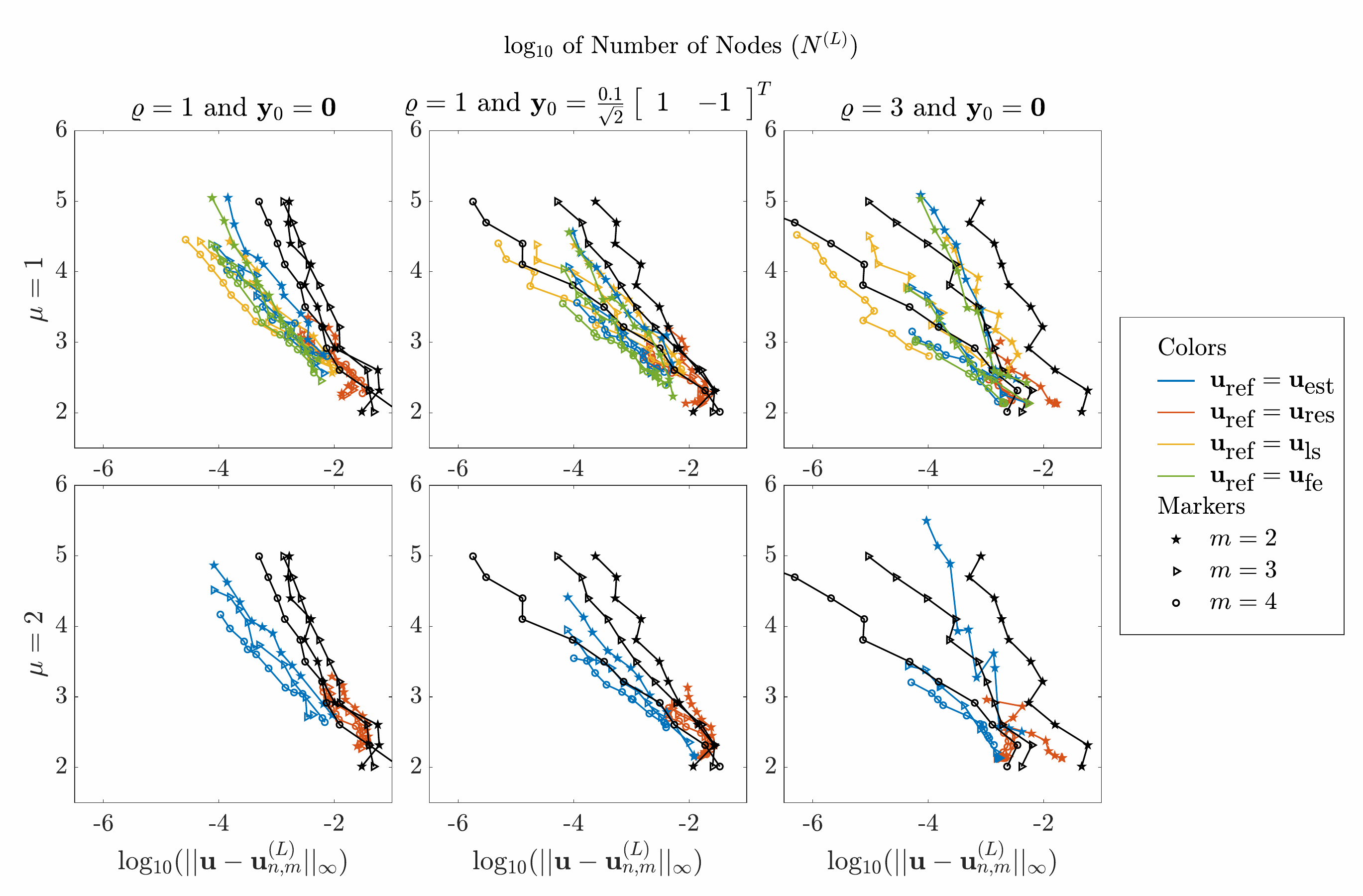}
    \caption{Base ten logarithm of number of nodes, $N^{(L)}$, required to compute the solution to \eqref{eq:PDEexact} with base ten logarithm of the forward error indicated by the horizontal axis for $u_{4}$ with various choices of $\varrho$ and $\mathbf{y}_{0}$.  Notice that the number of nodes required is consistent with respect to the refinement indicator.  However, in the left frames, a discontinuity in the derivatives with $\lvert\boldsymbol{\alpha}\rvert=1$ is at the origin.  This violates continuity assumptions required to express the dependence of \eqref{eq:forwarderror} and \eqref{eq:errorestimator} on the polynomial order, so that each indicator performs similiarly for all values of $m$ and $m+\mu$.  In the middle frames the discontinuity is shifted just outside the domain, so the dependence on $m$ becomes apparent.  The right frames illustrates the case where the discontinuity is then moved to derivatives with $\lvert\boldsymbol{\alpha}\rvert=2$, with the dependence on $m$ again becoming apparent.  In all frames, the black curves illustrate the relationship between the error and number of nodes when $N$ quasi-uniformly spaced nodes are used to solve the equation (i.e., there is no refinement). For each fixed choice of $m$ and $\mu$ that is shown, the adaptive methods all outperform uniform refinement over this range of errors.}
    \label{fig:CircleNumberofNodes}
\end{figure}

Finally, figure \ref{fig:CircleErrorRatios} serves to illustrate that the assumption that the smoothness of the solution, which impacts the radius of convergence of a power series of the solution, has an effect on the ability of the refinement indicator to estimate the forward error.  In these subplots, both the increase of $\varrho$ from $1$ to $3$ and the shift of $\mathbf{y}_{0}$ from the origin to $0.1/\sqrt{2}\left[\begin{array}{cc} 1 -1\end{array}\right]^{T}$ serve the purpose of increasing the smoothness (and ensuring the satisfaction of all assumptions) when considering $u_{4}$ on $\Omega_{2}$.  Clearly, the ability of $\mathbf{u}_{\mbox{ls}}$ to estimate the forward error is highly dependent on these changes.  On the other hand, $\mathbf{u}_{\mbox{est}}$ appears to adequately estimate the forward error in any case, while $\mathbf{u}_{\mbox{res}}$ is a poor estimator in all cases.

\begin{figure}
    \centering
    \includegraphics[width=\linewidth]{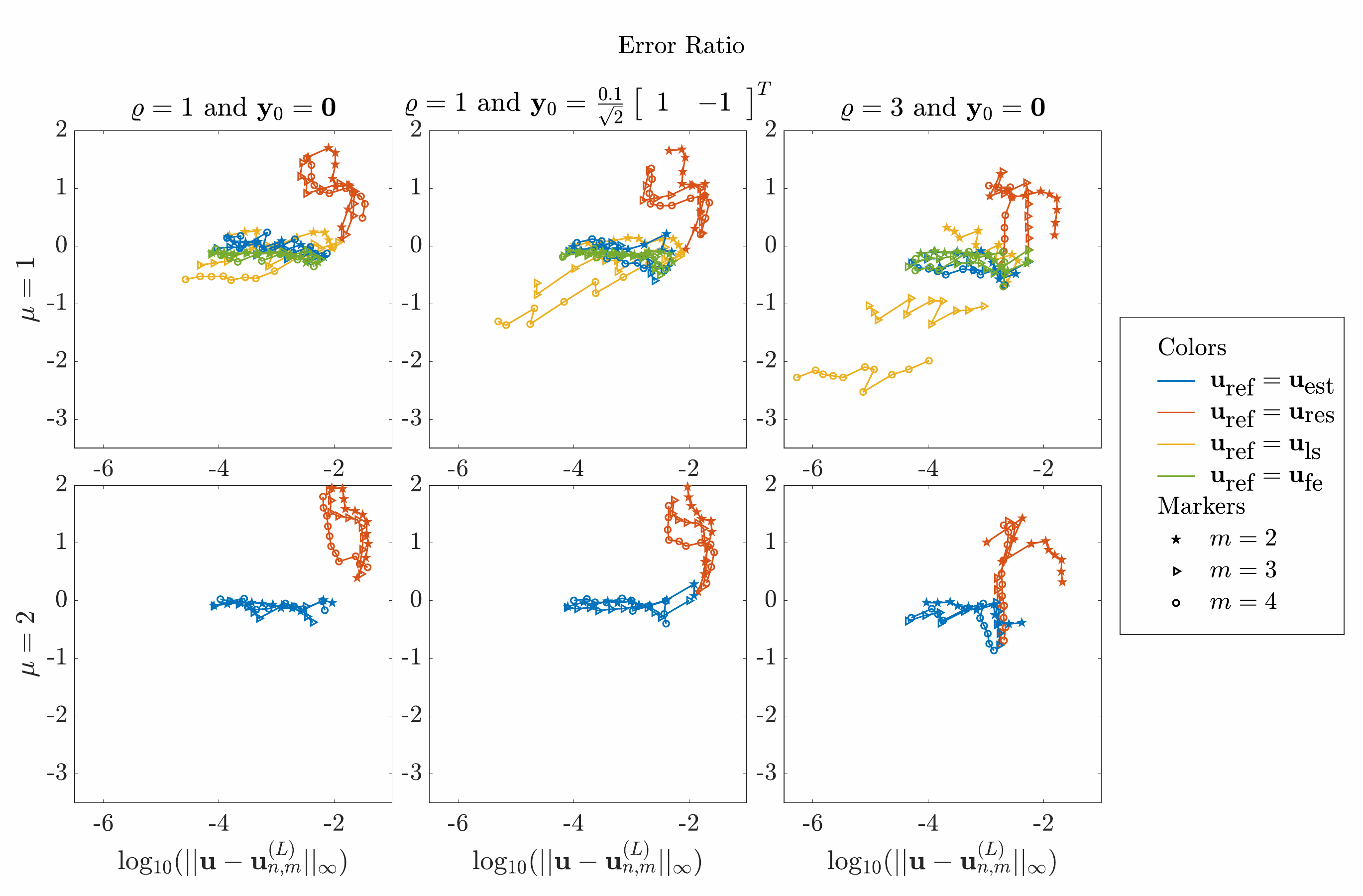}
    \caption{Base ten logarithm of the error ratio \eqref{eq:errorratio} when computing the solution to \eqref{eq:PDEexact} with base ten logarithm of the forward error indicated by the horizontal axis for $u_{4}$ with various choices of $\varrho$ and $\mathbf{y}_{0}$. Notice that for $\mathbf{u}_{\mbox{est}}$ and $\mathbf{u}_{\mbox{fe}}$ the error ratio approaches 1, i.e., the base ten logarithm approaches zero, as the forward error decreases.  This is not the case for either $\mathbf{u}_{\mbox{ls}}$ or $\mathbf{u}_{\mbox{res}}$.  Clearly, the ability of $\mathbf{u}_{\mbox{ls}}$ to estimate the forward error is highly dependent on the smoothness of the solution and location of its discontinuities.  On the other hand, $\mathbf{u}_{\mbox{est}}$ appears to adequately estimate the forward error in any case, while $\mathbf{u}_{\mbox{res}}$ is a poor estimator in all cases. }
    \label{fig:CircleErrorRatios}
\end{figure}

\begin{remark}
    Figures \ref{fig:SquareNumberofNodes}, \ref{fig:SquareComputationTimes} and \ref{fig:CircleNumberofNodes} also illustrate the relationship between the error and number of nodes/computation times when quasi-uniformly spaced nodes are used to solve the equation (i.e., there is no refinement).  For each fixed choice of $m$ and $\mu$ that is shown, the adaptive methods all outperform uniform refinement over this range of errors.
\end{remark}

\section{Conclusions} \label{sec:Conclusions}

This article presented an approach to approximating solutions to the Poisson equation, whereby the computational algorithm locally adapts the spacing of the discrete point set to achieve a prescribed tolerance on a local error estimator that is novel in relationship to kernel methods.  The use of kernel methods, in the context of RBF-FD, inspired the introduction of a node placement strategy designed to maintain the ``meshfree" properties of these approximations.  The computational experiments illustrate that existing refinement indicators provide similar performance to the estimator derived herein when considering the number of adaptively placed nodes that are required to achieve a prescribed tolerance on the indicator.  However, these existing indicators exhibit a complicated relationship with the forward error that can make it difficult to discern how well the actual solution is being approximated, even for low tolerances.  On the other hand, these same experiments have shown that the estimator, $\mathbf{u}_{\mbox{est}}$, and actual absolute error agree, particularly as the discrete node set is refined.

\printcredits

\bibliographystyle{cas-model2-names}

\bibliography{references}



\end{document}